\documentclass[reqno]{amsart}
\usepackage{amsmath}
\usepackage{amsthm}
\usepackage{graphicx}
\usepackage{amsfonts}
\usepackage{amssymb}
\usepackage{enumerate}
\numberwithin{equation}{section}

\newtheorem{theorem}{Theorem}[section]
\newtheorem{lemma}[theorem]{Lemma}
\newtheorem{proposition}[theorem]{Proposition}
\newtheorem{corollary}[theorem]{Corollary}

\theoremstyle{definition}

\newtheorem{definition}[theorem]{Definition}
\newtheorem{remark}[theorem]{Remark}

\newtheorem{innercustomthm}{Assumption}
\newenvironment{assumption}[1]
  {\renewcommand\theinnercustomthm{#1}\innercustomthm}
  {\endinnercustomthm}

\def\E{{\mathbb E}}
\def\R{{\mathbb R}}

\def\P{{\mathcal P}}
\def\RC{{\mathcal R}}
\def\B{{\mathcal B}}
\def\V{{\mathcal V}}

\def\M{{\mathcal M}}

\def\Q{{\mathcal Q}}

\def\F{{\mathcal F}}
\def\C{{\mathcal C}}

\title[Mean field games via controlled martingale problems]{Mean field games via controlled martingale problems: Existence of Markovian equilibria}
\author{Daniel Lacker}
\address{ORFE,  Princeton University,
Princeton, NJ  08544, USA.}
\email{dlacker@princeton.edu}
\thanks{Partially supported by NSF: DMS-0806591}


\begin{document}

\begin{abstract}
Mean field games are studied in the framework of controlled martingale problems, and general existence theorems are proven in which the equilibrium control is Markovian. The framework is flexible enough to include degenerate volatility, which may depend on both the control and the mean field. The objectives need not be strictly convex, and the mean field interactions considered are nonlocal and Wasserstein-continuous. When the volatility is nondegenerate, continuity assumptions may be weakened considerably. The proofs first use relaxed controls to establish existence. Then, using a convexity assumption and measurable selection arguments, strict (non-relaxed) Markovian equilibria are constructed from relaxed equilibria.
\end{abstract}

\maketitle

\section{Introduction}
The purpose of this paper is to develop a new framework for the analysis of (continuous-time) mean field games and to use it to prove several very general existence results. The story of mean field games, introduced independently in the pioneering work of Huang, Malham\'e, and Caines \cite{huangmfg1} and Lasry and Lions \cite{lasrylionsmfg}, begins with a certain class of large-population stochastic differential games. Agents $i=1,\ldots,n$ have private state processes $X^i$, the dynamics of which are given by the stochastic differential equation (SDE)
\begin{align*}
dX^i_t &= b(t,X^i_t,\bar{\mu}^n_t,\alpha^i_t)dt + \sigma(t,X^i_t,\bar{\mu}^n_t,\alpha^i_t)dW^i_t, \\
\bar{\mu}^n_t &= \frac{1}{n}\sum_{j=1}^n\delta_{X^j_t}.
\end{align*}
Here $W^1,\ldots,W^n$ are independent Wiener processes, $\bar{\mu}^n_t$ is the empirical measure of the state processes at time $t$, and $\alpha^i_t$ is the control process chosen by agent $i$. Agent $i$ seeks to choose a control $\alpha^i$ to maximize
\[
\E\left[\int_0^Tf(t,X^i_t,\bar{\mu}^n_t,\alpha^i_t)dt + g(X^i_T,\bar{\mu}^n_T)\right].
\]
The agents have the same state process coefficients and objective functions, and their optimization problems are coupled only through the empirical distribution $\bar{\mu}^n_t$ of the state processes. Since the objectives are coupled, we naturally look for Nash equilibria. If the initial conditions $X^1_0,\ldots,X^n_0$ are symmetric, then in a sense so is the entire game. If the number of agents $n$ is large, we hope to learn something about the Nash equilibria of this game from the corresponding \emph{mean field game}, which intuitively captures the idea of an infinite-agent version of the game. The structure of the mean field game is as follows: Fix a function $t \mapsto \mu_t$ with values in the space $\P(\R^d)$ of probability measures on $\R^d$, and solve (if possible) the optimal control problem given by
\begin{align}
\sup_\alpha \ &\E\left[\int_0^Tf(t,X^\alpha_t,\mu_t,\alpha_t)dt + g(X^\alpha_T,\mu_T)\right], \nonumber \\
\text{ s.t. } dX^\alpha_t &= b(t,X^\alpha_t,\mu_t,\alpha_t)dt + \sigma(t,X^\alpha_t,\mu_t,\alpha_t)dW_t. \label{introSDE}
\end{align}
Then, let $\Phi(\mu)$ denote the set of time-marginal laws of the optimally controlled state processes:
\[
\Phi(\mu) = \left\{(\text{Law}(X^{\alpha^*}_t))_{t \in [0,T]} : \alpha^* \text{ is optimal}\right\}.
\]
A mean field game (MFG) solution is a fixed point of this map $\Phi$, or $\mu \in \Phi(\mu)$. In short, the MFG problem consists of solving (a family of) control problems and finding a fixed point.

Intuitively, the SDE \eqref{introSDE} describes the state process dynamics of a single representative agent, and $\mu_t$ represents the distribution of an infinity of agents' state processes. The representative agent cannot influence $\mu_t$ and thus considers it as fixed when solving the optimization problem. If each agent among the infinity is identical and acts in the same way, then consistency demands that the time-marginal laws of the representative's optimally controlled state process must agree with $\mu_t$.

Three of the main questions in the theory of mean field games pertain to existence of solutions, uniqueness of solutions, and convergence of finite-player equilibria. The question of convergence is two-sided: one may try to show that the Nash equilibria of the finite-player games converge in some sense to a mean field limit, or one may try to use a MFG solution to construct \emph{approximate equilibria} for the finite-player games. The latter approach is far more common in the literature since \cite{huangmfg1}, whereas few results exist so far for the former (e.g. \cite[Theorem 2.3]{lasrylionsmfg}). Uniqueness is generally harder to come by but is known to hold under the monotonicity condition of Lasry and Lions \cite{lasrylionsmfg} or when the time horizon is small \cite{huangmfg1}.

The present paper studies solely the problem of existence, and the first main result, Theorem \ref{th:mainexistence-intro}, is stated precisely in Section \ref{se:statement}, immediately following the introduction. Theorem \ref{th:mainexistence-intro} is quite general in scope: there exists a MFG solution for which the corresponding optimal control happens to be \emph{Markovian} (though still optimal among non-Markovian controls). To comment on the assumptions: Both the control and mean field may influence both the drift and volatility coefficients. The volatility may degenerate, and so our results include first-order mean field games. The mean field dependence is nonlocal and continuous with respect to a $p$-Wasserstein distance, and $p$ also determines the growth rates of the data. In Section \ref{se:elliptic}, under the additional assumption that the volatility is uncontrolled and uniformly nondegenerate, an alternative existence result (Theorem \ref{th:ellipticexistence}) is proven under weaker continuity assumptions.

This paper unifies and generalizes several known existence results, but at least as important is the novel framework for the analysis. A fundamental difficulty in the analysis of mean field games is often deriving nice properties for the optimal feedback control $\alpha_t = \alpha(t,X_t,\mu_t)$, to prove that the fixed point map $\Phi$ described above is continuous (and single-valued). We avoid this issue entirely by working with \emph{relaxed controls}, which essentially compactifies the class of admissible controls and converts a stochastic optimal control problem into a linear program (albeit in infinite dimension). The choice variable is no longer a \emph{control processes} but rather a \emph{joint law of the control-state pair}. Existence is proven via Kakutani's fixed point theorem, by showing that the set of optimally controlled state process laws is suitably continuous, as a function of the input measure flow $\mu$. Our use of relaxed controls may be seen as a form of \emph{mixed strategies} in the stochastic differential setting, and our formulation allows us to bring classical game-theoretic arguments to bear on the mean field game problem. Indeed, since Nash \cite{nash-1950}, game theorists have exploited set-valued fixed point theory to prove the existence of equilibria with no need for \emph{unique} best responses. 

Introducing relaxed controls facilitates existence proofs, but of course we are more interested in MFG solutions involving \emph{strict} (non-relaxed) controls. The punchline is that (under an additional convexity assumption) from any relaxed control one can construct a superior strict Markovian control, and crucially this can be done \emph{without changing the time-marginal laws} of the state process. In particular, this construction does not disturb the fixed point property of MFG solutions, and we can construct a MFG solution with strict Markovian control from a MFG solution with relaxed control. This allows us to state our main existence result in Section \ref{se:statement} without any reference to relaxed controls.

The mean field games literature is dominated by two lines of research: one is based on partial differential equations (PDEs), and the other is based on the stochastic maximum principle. Following Lasry and Lions \cite{lasrylionsmfg,lasrylions-jeux1,lasrylions-jeux2,gueantlasrylionsmfg}, the PDE approach studies the control problems via the (backward) Hamilton-Jacobi-Bellman equation and resolves the fixed point by coupling the HJB with the (forward) Kolmogorov equation of the state process. The successes of PDE methods include its amenability to numerics \cite{achdou-mfgnumerics} and its ability to handle \emph{local} mean field interactions \cite{lasrylionsmfg,gomes-timedependent-subquadratic,gomes-timedependent-superquadratic,cardaliaguet-mfgweaksolutionsfirstorder}, in which the functions of $(X_t,\mu_t)$ involves the density $d\mu_t/dx(X_t)$, but these matters are not addressed in this paper. For good surveys, see the notes of Cardaliaguet \cite{cardaliaguet-mfgnotes} and Gomes and Sa\'ude \cite{gomessaude-mfgsurvey}.

On the other hand, the stochastic maximum principle reduces the control problems to coupled forward-backward stochastic differential equations (FBSDEs), and the fixed point condition injects an additional feedback into the system in the sense that the coefficients of the FBSDE now depend on the law of the solution. The mean field FBSDE (as it was named in \cite{carmonadelarue-mffbsde}) underlies the mean field game analysis of Carmona and Delarue \cite{carmonadelarue-mfg}, Bensoussan et al. \cite{bensoussan-lqmfg}, and the recent book \cite{bensoussan-mfgbook}. This stochastic approach leads to a notably clean solution of linear-quadratic mean field games, as in \cite{bensoussan-lqmfg,carmonadelaruelachapelle-mkvvsmfg}, as well as an interesting generalization of mean field games involving one dominant agent competing with a mean field of minor agents \cite{huang-majorminor,buckdahn-majorminormfg}.

A bit of an outlier is the probabilistic weak formulation of \cite{carmonalacker-probabilisticweakformulation}, which employs a well-known Girsanov transformation of the control problems under the crucial assumption that the volaility is nondegenerate and uninfluenced by the mean field or the control. Of the papers mentioned so far, \cite{carmonalacker-probabilisticweakformulation} is perhaps philosophically closest to the approach of the present paper in two ways: It works directly with the \emph{law} of the controlled state process, rather than with the process itself, and it appears to be the only work on mean field games thusfar which allows for non-unique optimal controls (by appealing to a set-valued fixed point theorem). A superficial analogy could be made between the present work and that of Kolokoltsov et al. \cite{kolokoltsov-mfg}, which similarly works first and foremost with the infinitesimal generator of the state process, but the similarity ends here: \cite{kolokoltsov-mfg} is really a generalization of the PDE approach.

Our use of relaxed controls for stochastic optimal control problems borrows heavily from the fundamental work of El Karoui et al. \cite{elkaroui-compactification} and the more general arguments of Haussmann and Lepeltier \cite{haussmannlepeltier-existence}. The argument in \cite{elkaroui-compactification,haussmannlepeltier-existence} for producing a Markovian control from a relaxed control follows Krylov's ideas (explained in \cite[Chapter 12]{stroockvaradhanbook}), but we prefer to exploit the recent ``mimicking theorem'' of Brunick and Shreve \cite{brunickshreve-mimicking}, which generalizes a well-known result of Gy\"ongy \cite{gyongy-mimicking}. Borkar and Ghosh \cite{borkarghosh-games} provide what appears to be the only study of $n$-player stochastic differential games using relaxed controls, and our use of Markovian relaxed controls in Section \ref{se:elliptic} resembles theirs.

Given the success of relaxed control theory in quite general stochastic optimal control problems, several extensions of our framework to various mean field game problems are feasible. An extension of the framework will appear in follow-up work addressing the convergence of finite-player equilibria as well as mean field games with \emph{common noise} (which have only appeared so far in some particular models \cite{gueantlasrylionsmfg,carmonafouque-systemicrisk}). Presumably more straightforward would be an adaptation of our framework to mean field games involving more general time horizons (e.g. infinite or up to a stopping time, as in \cite{haussmannlepeltier-existence}) or state processes given by jump-diffusions (as in \cite[Section 8]{elkaroui-compactification}). Relaxed control theory is pushed much further by Kurtz and Stockbridge in \cite{kurtzstockbridge-1998}, but their level of abstraction seems a bit out of reach at the moment; see Section \ref{se:extensions} for further comments.

It is worth mentioning that our existence theorems, as with most of those obtained via relaxed control theory, are rather abstract in nature and provide little insight into how to compute MFG solutions. But this nonconcreteness is quite prevalent in MFG existence theory: while the PDE and stochastic methods described above are more tangible in their handling of the control problems, the construction of MFG solutions is still through abstract and rather intractable Schauder-type fixed point theorems (except when the time horizon is small and contraction arguments are available).

The paper is structured as follows. The main assumptions and existence theorem are stated precisely in Section \ref{se:statement}, without any mention of relaxed controls. Section \ref{se:relaxed} introduces relaxed controls and the language of controlled martingale problems, in terms of which the mean field game is then described. Theorem \ref{th:markovselection} gives the procedure for constructing non-relaxed and Markovian solutions from relaxed solutions, and this is proven immediately in Section \ref{se:relaxed}. The proof of existence of relaxed MFG solutions is split between Section \ref{se:boundedcoefficients} and \ref{se:unboundedcoefficients}. Section \ref{se:boundedcoefficients} treats the simpler case of bounded state coefficients and control space, while Section \ref{se:unboundedcoefficients} extends this to the unbounded case by an approximation procedure. Some refinements are discussed in Section \ref{se:elliptic}, when the volatility is uncontrolled and nondegenerate. A simple but important counterexample is discussed in Section \ref{se:counterexample}, and finally Section \ref{se:extensions} points to potential future work. Some technical results and background on Wasserstein spaces are gathered in the appendix.

\section{Statement of main results} \label{se:statement}
For a measurable space $(\Omega,\F)$, let $\P(\Omega)$ denote the set of probability measures on $(\Omega,\F)$. When $\Omega$ is a topological space, let $\B(\Omega)$ denote its Borel $\sigma$-field, and endow $\P(\Omega)$ with the topology of weak convergence. Fix a finite time horizon $T > 0$. Let $\C^k = C([0,T];\R^k)$ denote the set of continuous functions from $[0,T]$ to $\R^k$, endowed with the supremum norm $\|\cdot\|_T$, where
\[
\|x\|_t := \sup_{s \in [0,t]}|x_s|, \ t \in [0,T], \ x \in \C^k.
\]
For $\mu \in \P(\C^k)$, let $\mu_t$ denote the image of $\mu$ under the map $\C^k \ni x \mapsto x_t \in \R^k$. For $p \ge 0$ and a separable metric space $(E,d)$, let $\P^p(E)$ denote the set of $\mu \in \P(E)$ satisfying $\int_Ed^p(x,x_0)\mu(dx) < \infty$ for some (and thus for any) $x_0 \in E$. For $p \ge 1$ and $\mu,\nu \in \P^p(E)$, let $d_{E,p}$ denote the $p$-Wasserstein distance, given by
\begin{align}
d_{E,p}(\mu,\nu) := \inf\left\{\int_{E \times E}\gamma(dx,dy)d^p(x,y) : \gamma \in \P(E \times E) \text{ has marginals } \mu,\nu\right\}^{1/p}. \label{def:wasserstein}
\end{align}
Unless otherwise stated, the space $\P^p(E)$ is always equipped with the metric $d_{E,p}$. If $E$ is complete and separable, so is $(\P^p(E),d_{E,p})$. Appendix \ref{ap:wasserstein} complies some background and technical results regarding Wasserstein distances. 
Given $p \ge 0$ and $\mu \in \P(\R^k)$, we will make frequent use of the abbreviation $|\mu|^p := \int_{\R^k}|x|^p\mu(dx)$. Similarly, for $\mu \in \P(\C^k)$ we will write
\begin{align}
\|\mu\|^p_t := \int_{\C^k}\|x\|^p_t\mu(dx), \text{ for } \mu \in \P^p(\C^k). \label{def:measureconvention}
\end{align}
These abbreviations are avoided whenever they may cause confusion; for example, the reader will not be required to parse such expressions as $(|\mu|^p)^{1/p}$, which could conceivably stand for $(\int|z|^p\mu(dz))^{1/p}$.

The mean field game is specified by the following data. Let $A$ denote the control space, let $\lambda \in \P(\R^d)$ denote the initial state distribution, and let $p \ge 1$.
\begin{align*}
(b,\sigma,f) &: [0,T] \times \R^d \times \P^p(\R^d) \times A \rightarrow \R^d \times \R^{d \times m} \times \R \\ 
g &: \R^d \times \P^p(\R^d) \rightarrow \R
\end{align*}
We are given two additional constants $p',p_\sigma \ge 0$, the role of which will soon be clear.
The existence result is subject to the following main assumptions.

\begin{assumption}{\textbf{(A)}} \label{assumption:A} {\ }
\begin{enumerate}
\item[(A.1)] The functions $b$, $\sigma$, $f$, and $g$ of $(t,x,\mu,a)$ are measurable in $t$ and continuous in $(x,\mu,a)$.
\item[(A.2)] There exists $c_1 > 0$ such that, for all $(t,\mu,a) \in [0,T] \times \P^p(\R^d) \times A$ and all $x,y \in \R^d$,
\begin{align*}
|b(t,x,\mu,a) - b(t,y,\mu,a)| + |\sigma(t,x,\mu,a) - \sigma(t,y,\mu,a)| &\le c_1|x-y|,
\end{align*}
and
\begin{align*}
|b(t,x,\mu,a)| &\le c_1\left[1 + |x| + \left(\int_{\R^d}|z|^p\mu(dz)\right)^{1/p} + |a|\right], \\
|\sigma\sigma^\top(t,x,\mu,a)| &\le c_1\left[1 + |x|^{p_\sigma} + \left(\int_{\R^d}|z|^p\mu(dz)\right)^{p_\sigma/p} + |a|^{p_\sigma}\right]
\end{align*}
\item[(A.3)] There exist $c_2, c_3 > 0$ such that, for each $(t,x,\mu,a) \in [0,T] \times \R^d \times \P^p(\R^d) \times A$, 
\begin{align*}
|g(x,\mu)| &\le c_2\left(1 + |x|^p + |\mu|^p\right), \\
-c_2\left(1 + |x|^{p} + |\mu|^p + |a|^{p'}\right) \le f(t,x,\mu,a) &\le c_2\left(1 + |x|^p + |\mu|^p\right) - c_3|a|^{p'}
\end{align*}
\item[(A.4)] The control space $A$ is a closed subset of a Euclidean space. (More generally, as in \cite{haussmannlepeltier-existence}, a closed $\sigma$-compact subset of a Banach space would suffice.)
\item[(A.5)] The initial distribution $\lambda$ is in $\P^{p'}(\R^d)$, and the exponents satisfy $p' > p \ge 1 \vee p_\sigma$ and $p_\sigma \in [0,2]$.
\end{enumerate}
\end{assumption}

A typical case is $p' = 2$, $p=1$, and $p_\sigma=0$ (i.e. $\sigma$ bounded). Along the way, we will also treat the situation of compact $A$ and bounded $b$ and $\sigma$. Unfortunately, our assumptions to not cover all linear-quadratic models. When the objective $f$ is quadratic in the control $a$, we are forced to choose $p'=2$, and the constraint $p < p'$ forces $f$ and $g$ to be strictly subquadratic in $x$. However, this is not surprising in light of the counterexample discussed in Section \ref{se:counterexample}.

Our main existence theorem for mean field games requires one additional assumption, familiar in relaxed control theory ever since Filippov's work \cite{filippov-convexity}. Without this assumption, a form of the following main result involving relaxed controls still holds, but its statement requires additional technical developments which we postpone to Section \ref{se:relaxed}.

\begin{assumption}{\textbf{(Convex)}} \label{assumption:convex} 
For each $(t,x,\mu) \in [0,T] \times \R^d \times \P^p(\R^d)$, the subset
\[
K(t,x,\mu) := \left\{\left(b(t,x,\mu,a),\sigma\sigma^\top(t,x,\mu,a),z\right) : a \in A, \ z \le f(t,x,\mu,a)\right\}
\]
of $\R^d \times \R^{d \times d} \times \R$ is convex.
\end{assumption}

\begin{theorem} \label{th:mainexistence-intro}
Suppose assumptions \ref{assumption:A} and \ref{assumption:convex} hold. Then there exist $\mu \in \P^p(\C^d)$ and a measurable function $\hat{\alpha} : [0,T] \times \R^d \rightarrow A$ satisfying the following:
\begin{enumerate}
\item There exists a filtered probability space $(\Omega,\F_t,P)$ supporting an $m$-dimensional $\F_t$-Brownian motion $W$ and a $d$-dimensional $\F_t$-adapted process $X$ such that 
\begin{align*}
\begin{cases}
dX_t &= b(t,X_t,\mu_t,\hat{\alpha}(t,X_t))dt + \sigma(t,X_t,\mu_t,\hat{\alpha}(t,X_t))dW_t, \quad P \circ X_0^{-1} = \lambda, \\
P \circ X^{-1} &= \mu.
\end{cases}
\end{align*}
\item Suppose $(\Omega',\F'_t,P')$ is another filtered probability space supporting an $m$-dimensional $\F'_t$-Brownian motion $W'$, a $d$-dimensional $\F'_t$-adapted process $X'$, and a $\F'_t$-adapted process $\alpha'$ such that 
\begin{align*}
dX'_t &= b(t,X'_t,\mu_t,\alpha'_t)dt + \sigma(t,X'_t,\mu_t,\alpha'_t)dW'_t, \quad P \circ (X'_0)^{-1} = \lambda.
\end{align*}
Then
\begin{align*}
\E^P&\left[\int_0^Tf(t,X_t,\mu_t,\hat{\alpha}(t,X_t))dt + g(X_T,\mu_T)\right] \\
	&\ge \E^{P'}\left[\int_0^Tf(t,X'_t,\mu_t,\alpha'_t)dt + g(X'_T,\mu_T)\right].
\end{align*}
\end{enumerate}
\end{theorem}

\begin{remark}
A typical case of assumption \ref{assumption:convex} is when $\sigma$ is uncontrolled, $A$ is convex, the drift $b$ is affine in $a$, and $f$ is concave in $a$. Examples are somewhat less natural when the volatility is controlled, but we highlight some simple cases. A first example is when $A \subset \R^{d' \times m}$ is convex, $\sigma(t,x,\mu,a)=\tilde{\sigma}(t,x,\mu)a$ where $\tilde{\sigma}$ takes values in $\R^{d \times d'}$, $b$ is affine in $aa^\top$, and $f$ is concave in $aa^\top$. For a second example, suppose the drift and volatility are controlled \emph{separately}, in the sense that $A = A_1 \times A_2$, $b=b(t,x,\mu,a_1)$, and $\sigma=\sigma(t,x,\mu,a_2)$; then, assumption \ref{assumption:convex} holds if $b$ is affine in $a_1$, $\sigma$ is linear in $a_2$, and $f=f(t,x,\mu,(a_1,a_2))$ is concave in $(a_1,a_2a_2^\top)$. Of course, many non-affine examples exist, but they are not as easily summarized.
\end{remark}

\begin{remark}
There is a natural formulation of this result in terms of the following forward-backward PDE system, which, following Lasry and Lions \cite{lasrylionsmfg}, is often taken as the definition of a mean field game:
\begin{align*}
\begin{cases}
&-\partial_tv(t,x) -  H\left(t,x,\mu_t,Dv(t,x),D^2v(t,x)\right) = 0, \\
&\partial_t\mu_t(x) - \hat{L}^*[t,\mu_t]\mu_t(x) = 0, \\
&\mu_0 =\lambda, \quad v(T,x) = g(x,\mu_T),
\end{cases}
\end{align*}
where the Hamiltonian $H$ is defined by
\begin{align*}
H(t,x,\nu,y,z) &:= \sup_{a \in A} \left[y^\top b(t,x,\nu,a) + \frac{1}{2}\mathrm{Tr}[\sigma\sigma^\top(t,x,\nu,a)z] + f(t,x,\nu,a)\right],
\end{align*}
and, for each $(t,\nu)$, $\hat{L}^*[t,\nu]$ is the formal adjoint of the operator
\begin{align*}
\hat{L}[t,\nu]\phi(x) := b(t,x,\nu,\hat{\alpha}(t,x))^\top D\phi(x) + 
\frac{1}{2}\text{Tr}\left[\sigma\sigma^\top(t,x,\nu,\hat{\alpha}(t,x))D^2\phi(x)\right].
\end{align*}
Intuitively, the function $v(t,x)$ is the value function coming from the stochastic control problem faced in equilibrium by a representative agent, and the measure flow $\mu_t$ gives the time-$t$ distribution of a continuum of (independent, identically distributed) agents' state processes.
Under various assumptions one can conclude from our Theorem \ref{th:mainexistence-intro} that there exists a solution of this PDE system, where the forward equation for $\mu_t$ is solved in a weak sense and the backward equation for $v$ is either a classical or viscosity solution. One can then simplify the form of the equations by arguing that $\hat{\alpha}(t,x)$ should attain the supremum in $H(t,x,\mu_t,Dv(t,x),D^2v(t,x))$. Essentially, this depends only on being able to solve the HJB part of the equation when $\mu_t$ is treated as fixed.
\end{remark}

\section{The relaxed mean field game} \label{se:relaxed}
As is common when studying weak solutions of stochastic equations, we reformulate the mean field game problem of Theorem \ref{th:mainexistence-intro} on a canonical probability space.

\subsection{Relaxed controls}
Let $\V[A]$ denote the set of measures $q$ on $[0,T] \times A$ with first marginal equal to Lebesgue measure (i.e. $q([s,t] \times A) = t-s$ for $0 \le s < t \le T$) such that
\[
\int_{[0,T] \times A}q(dt,da)|a|^p < \infty.
\]
When $A$ is understood, we write simply $\V$. An element of $\V$ is called a \emph{relaxed control}. Endow $\V$ with the $p$-Wasserstein metric, adapted naturally from \eqref{def:wasserstein} as follows:
\begin{align}
d_{\V[A]}(q^1,q^2) := d_{[0,T] \times A,p}(q^1/T,q^2/T). \label{def:Vmetric}
\end{align}
This renders $\V$ a complete separable metric space, and when $A$ is compact so is $\V[A]$. We will frequently identify an element $q \in \V$ with the measurable map $t \mapsto q_t \in \P(A)$ arising from its disintegration $q(dt,da) = dtq_t(da)$, which is unique up to (Lebesgue) almost everywhere equality. Of particular interest are the \emph{strict controls}, which are of the form $q = dt\delta_{\alpha(t)}(da)$ for measurable $\alpha : [0,T] \rightarrow A$. Let $\Omega[A] := \V[A] \times \C^d$, endowed with its Borel $\sigma$-field, and again we abbreviate this to $\Omega$ when $A$ is understood. A generic element of $\Omega$ is denoted $(q,x)$, and the identity maps on $\V$ and $\C^d$ are denoted $\Lambda$ and $X$, respectively. Consider the filtrations
\[
\F^\Lambda_t := \sigma(1_{[0,t]}\Lambda) = \sigma\left(\Lambda(C) : C \in \B([0,t] \times A)\right)
\]
on $\V$ and $\F^X_t := \sigma(X_s : s \le t)$ on $\C^d$, along with the product $\F_t := \F^\Lambda_t \otimes \F^X_t$ defined on $\Omega$. The following notational convention will be used occasionally without mention. Given any spaces $E$, $E'$, and $F$ and any function $\phi : E \rightarrow F$, the same symbol $\phi$ will denote the natural extension of the function to $E \times E'$ given by $\phi(e,e') := \phi(e)$. In this way, $X$ is a process on both $\C^d$ and $\Omega$.

\begin{remark}
As with $\V$ and $\Omega$, much of the notation introduced below naturally depends on the data $(b,\sigma,f,g,A)$. The proof of Theorem \ref{th:mainexistence-intro} is done first for bounded coefficients and compact control space $A$, and the general case is proven by approximation. Thus it will be useful later to keep track of this dependence.
\end{remark}

We state here for future reference a reassuring technical lemma which will be useful in proving Theorem \ref{th:markovselection}. This lemma seems to be known and often used implicitly, but we sketch the proof for the reader's convenience, as a precise reference is difficult to locate.

\begin{lemma} \label{le:canonicalprocess}
There exists a $\F^{\Lambda}_t$-predictable process $\overline{\Lambda} : [0,T] \times \V \rightarrow \P(A)$ such that, for each $q \in \V$, $\overline{\Lambda}(t,q) = q_t$ for almost every $t \in [0,T]$. In particular, $q = dt\overline{\Lambda}(t,q)(da)$ for each $q \in \V$.
\end{lemma}
\begin{proof}
Define $F_\epsilon : [0,T] \times \V \rightarrow \P(A)$ by
\[
F_\epsilon(t,q) :=  \frac{1}{t - (t-\epsilon)^+}\int_{(t-\epsilon)^+}^tq_sds.
\]
Then $F_\epsilon(\cdot,q)$ is continuous for each $q \in \V$, and $F_\epsilon(t,\cdot)$ is $\F^\Lambda_t$-measurable for each $t \in [0,T]$. Hence $F_\epsilon$ is $\F^\Lambda_t$-predictable. Fix arbitrarily $q^0 \in \P(A)$, and for each $(t,q) \in [0,T] \times \V$ define
\[
\overline{\Lambda}(t,q) := \begin{cases}
\lim_{\epsilon \downarrow 0}F_\epsilon(t,q) & \text{ if the limit exists}, \\
q^0 &\text{ otherwise}.
\end{cases}
\]
Then $\overline{\Lambda}$ is predictable, and it follows from Lebesgue's differentiation theorem (arguing with a countable convergence-determining class of functions on $A$) that $\overline{\Lambda}(t,q) = q_t$ for almost every $t \in [0,T]$.
\end{proof}

We will abuse notation somewhat by writing $\Lambda_t = \overline{\Lambda}(t,\cdot)$ for the canonical process on $\V$ given by Lemma \ref{le:canonicalprocess}. This way, $\Lambda = dt\Lambda_t(da)$.

\subsection{Controlled martingale problems and MFG solutions}
The controlled state process will be described by way of its infinitesimal generator. Let $C^\infty_0(\R^d)$ denote the set of infinitely differentiable functions $\phi : \R^d \rightarrow \R$ with compact support, and let $D\phi$ and $D^2\phi$ denote the gradient and Hessian of $\phi$, respectively. Define the generator $L = L[b,\sigma,A]$ on $\phi \in C^\infty_0(\R^d)$ by
\[
L\phi(t,x,\mu,a) = b(t,x,\mu,a)^\top D\phi(x) + 
\frac{1}{2}\text{Tr}\left[\sigma\sigma^\top(t,x,\mu,a)D^2\phi(x)\right],
\]
for $(t,x,\mu,a) \in [0,T] \times \R^d \times \P^p(\R^d) \times A$. For $\phi \in C^\infty_0(\R^d)$ and $\mu \in \P^p(\C^d)$, define $M^{\mu,\phi}_t = M^{\mu,\phi}_t[b,\sigma,A] : \Omega \rightarrow \R$ by
\[
M^{\mu,\phi}_t(q,x) := \phi(x_t) - \int_{[0,t] \times A}q(ds,da)L\phi(s,x_s,\mu_s,a).
\]
Define the objective functional $\Gamma^\mu = \Gamma^\mu[f,g,A] : \Omega \rightarrow \R$ by
\[
\Gamma^\mu(q,x) := g(x_T,\mu_T) + \int_{[0,T] \times A}q(dt,da)f(t,x_t,\mu_t,a).
\]
\begin{definition}
For a measure $\mu \in \P^p(\C^d)$, let $\RC[b,\sigma,A](\mu)$ denote the set of $P \in \P(\Omega[A])$ satisfying the following: 
\begin{enumerate}
\item $P \circ X_0^{-1} = \lambda$
\item $\E^P\int_0^T|\Lambda_t|^pdt < \infty$.
\item $M^{\mu,\phi} = (M^{\mu,\phi}_t)_{t \in [0,T]}$ is a $P$-martingale for each $\phi \in C^\infty_0(\R^d)$.
\end{enumerate}
\end{definition}

As before, we abbreviate $\RC[b,\sigma,A](\mu)$ to $\RC(\mu)$ when the data is clear; this is the set of admissible joint laws of control-state pairs $(\Lambda,X)$. Define $J = J[f,g,A] : \P^p(\C^d) \times \P^p(\Omega) \rightarrow \R \cup \{-\infty\}$ and $\RC^* = \RC^*[b,\sigma,f,g,A] : \P^p(\C^d) \rightarrow 2^{\P(\Omega)}$ by
\begin{align*}
J(\mu,P) &:= \int_{\Omega}\Gamma^\mu\,dP, \\
\RC^*(\mu) &:= \arg\max_{P \in \RC(\mu)}J(\mu,P).
\end{align*}
Note that when $\mu \in \P^p(\C^d)$ and $P \in \P^p(\Omega)$, the upper bounds on $f$ and $g$ of assumption (A.3) ensure that the positive part of $\Gamma^\mu$ is $P$-integrable. Hence, $J$ is well-defined.
Using the growth assumptions (A.2) on the coefficients $(b,\sigma)$, Lemma \ref{le:stateestimate} below shows that $\RC(\mu) \subset \P^p(\Omega)$ for each $\mu \in \P^p(\C^d)$, so that $\RC^*(\mu)$ is also well-defined. 
A priori, $\RC^*(\mu)$ may be empty.

We say $P \in \P^p(\Omega)$ is a \emph{relaxed mean field game (MFG) solution} if $P \in \RC^*(P \circ X^{-1})$. We may also refer to the measure $P \circ X^{-1}$ on $\C^d$ itself as a relaxed MFG solution.
In other words, a relaxed MFG solution can be seen as a fixed point of the set-valued map
\[
\P^p(\C^d) \ni \mu \mapsto \left\{P \circ X^{-1} : P \in \RC^*(\mu)\right\} \in 2^{\P^p(\C^d)}.
\]
We say a measure $P \in \P(\Omega)$ \emph{corresponds to a strict control} if its $\V$-marginal is concentrated on the set of strict controls; that is, there exists an $\F_t$-progressively measurable $A$-valued process $\alpha_t$ on $\Omega$ such that $P(\Lambda = dt\delta_{\alpha_t}) = 1$.
On the other hand, $P$ \emph{corresponds to a relaxed Markovian control} if there exists a measurable function $\hat{q} : [0,T] \times \R^d \rightarrow \P(A)$ such that $P(\Lambda = dt\hat{q}(t,X_t)(da)) = 1$.
Finally, $P$ \emph{corresponds to a strict Markovian control} if there exists a measurable function $\hat{\alpha} : [0,T] \times \R^d \rightarrow A$ such that $P(\Lambda = dt\delta_{\hat{\alpha}(t,X_t)}(da)) = 1$.
If a relaxed MFG solution $P$ corresponds to a relaxed Markovian (resp. strict Markovian) control, then we say $P$ is a \emph{relaxed Markovian MFG soluiton} (resp. \emph{strict Markovian MFG solution}).

\begin{remark} \label{re:objectivegeneralization0}
In fact, the existence theorem for relaxed MFG solutions, Theorem \ref{th:relaxedexistence}, can be extended to include more general objective structures, such as risk-sensitive or mean-variance objectives. See Remark  \ref{re:objectivegeneralization2}. For the sake of simplicity, we stick with the more standard running-terminal objective structure.
\end{remark}

It is sometimes more convenient to represent $\RC(\mu)$ in terms of stochastic differential equations. To do this in general with control in the volatility requires some use of martingale measures. The uninitiated reader is referred to Walsh's monograph \cite{walsh-introspde} for a careful treatment, although the results we need are all found in the paper of El Karoui and M\'el\'eard \cite{elkarouimeleard-martingalemeasure}, and we use the precise terminology of the latter paper. If one is willing to assume $\sigma$ is uncontrolled, then there is no need for martingale measures, and one may replace $N(da,dt)$ with $dW_t$ in the following proposition.

\begin{proposition}[Theorem IV-2 of \cite{elkarouimeleard-martingalemeasure}] \label{pr:sderepresentation}
For $\mu \in \P^p(\C^d)$, $\RC(\mu)$ is precisely the set of laws $P' \circ (\Lambda,X)^{-1}$, where:
\begin{enumerate}
\item $(\Omega',\F'_t,P')$ is a filtered probability space supporting a $d$-dimensional $\F'_t$-adapted process $X$ as well as $m$ orthogonal $\F'_t$-martingale measures $N = (N^1,\ldots,N^m)$ on $A \times [0,T]$, each with intensity $\Lambda_t(da)dt$.
\item $P' \circ X_0^{-1} = \lambda$.
\item $\E^{P'}\int_0^T|\Lambda_t|^pdt < \infty$.
\item The state equation holds:
\begin{align}
dX_t &= \int_Ab(t,X_t,\mu_t,a)\Lambda_t(da)dt + \int_A\sigma(t,X_t,\mu_t,a)N(da,dt). \label{def:SDE}
\end{align}
\end{enumerate}
\end{proposition}

On any filtered probability space satisfying (1-3) of Proposition \ref{pr:sderepresentation}, the Lipschitz and growth assumptions of \ref{assumption:A} ensure that there exists a unique strong solution of \eqref{def:SDE}.

\subsection{Main results}
The following are the main results of the paper, with Theorem \ref{th:mainexistence-intro} following from Corollary \ref{co:strictmarkovexistence}. The rest of the section contains the proof of Theorem \ref{th:markovselection}, while Sections \ref{se:boundedcoefficients} and \ref{se:unboundedcoefficients} are devoted to proving Theorem \ref{th:relaxedexistence}.

\begin{theorem} \label{th:relaxedexistence}
Under assumption \ref{assumption:A}, there exists a relaxed MFG solution.
\end{theorem}

\begin{theorem} \label{th:markovselection}
Suppose \ref{assumption:A} holds. Let $\mu \in \P^p(\C^d)$ and $P \in \RC(\mu)$. Then there exist a measurable function $\hat{q} : [0,T] \times \R^d \rightarrow \P(A)$ and $P_0 \in \RC(\mu)$ such that:
\begin{enumerate}
\item $P_0(\Lambda = dt\hat{q}(t,X_t)(da)) = 1$. %$P_0  = P_0 \circ (dt\hat{q}(t,X_t)(da),X)^{-1}$.
\item $J(\mu,P_0) \ge J(\mu,P)$.
\item $P_0 \circ X_t^{-1} = P \circ X_t^{-1}$ for all $t \in [0,T]$.
\end{enumerate}
If also \ref{assumption:convex} holds, we can choose $\hat{q}$ of the form $\hat{q}(t,x) = \delta_{\hat{\alpha}(t,x)}$, for some measurable function $\hat{\alpha} : [0,T] \times \R^d \rightarrow A$.
\end{theorem}

In words, Theorem \ref{th:markovselection} says that for any control there exists a Markovian control (1) producing a greater reward (2) without altering the marginal distributions of the state process (3). When \ref{assumption:convex} holds, the new Markovian control can also be taken to be \emph{strict}.

\begin{corollary} \label{co:strictmarkovexistence}
Under assumption \ref{assumption:A}, there exists a relaxed Markovian MFG solution. Under assumptions \ref{assumption:A} and \ref{assumption:convex}, there exists a strict Markovian MFG solution.
\end{corollary}
\begin{proof}
Let $P$ be a relaxed MFG solution. Let $P_0$ be as in Theorem \ref{th:markovselection}. Since $P \in \RC^*(\mu)$ and $J(\mu,P_0) \ge J(\mu,P)$, we have $P_0 \in \RC^*(\mu)$. Let $\mu^0 := P_0 \circ X^{-1}$. Then $\mu^0_t = P_0 \circ X_t^{-1} = P \circ X_t^{-1} = \mu_t$ for all $t \in [0,T]$, and it follows that $\RC(\mu) = \RC(\mu^0)$, $J(\mu^0,\cdot) \equiv J(\mu,\cdot)$, and $\RC^*(\mu) = \RC^*(\mu^0)$. Thus $P_0 \in \RC^*(\mu^0)$.
\end{proof}

\begin{proof}[Proof of Theorem \ref{th:markovselection}]
As in \cite[Theorem 2.5(a)]{elkaroui-compactification}, we may find $\bar{m}$ and a measurable function $\bar{\sigma} : [0,T] \times \R^d \times \P^p(\R^d) \times \P(A) \rightarrow \R^{d \times \bar{m}}$ such that $\bar{\sigma}(t,x,\mu,q)$ is continuous in $(x,\mu,q)$ for each $t$,
\[
\bar{\sigma}\bar{\sigma}^\top(t,x,\mu,q) = \int_Aq(da)\sigma\sigma^\top(t,x,\mu,a),
\]
and $\bar{\sigma}(t,x,\mu,\delta_a) = \sigma(t,x,\mu,a)$ for each $(t,x,\mu,a)$; moreover, we may find a filtered probability space $(\Omega^1,\F^1_t,Q_1)$ supporting a $\bar{m}$-dimensional $\F^1_t$-Wiener process $W$, a $\R^d$-valued $\F^1_t$-adapted process $X^1$, and a $\P(A)$-valued process $\Lambda_t$ such that
\begin{align}
dX^1_t &= \int_Ab(t,X^1_t,\mu_t,a)\Lambda_t(da)dt + \bar{\sigma}(t,X^1_t,\mu_t,\Lambda_t)dW_t, \text{ and } \nonumber \\
P &= Q_1 \circ (dt\Lambda_t(da),X^1)^{-1}. \label{pf:markov0}
\end{align}
We claim that there exists a (jointly) measurable function $\hat{q} : [0,T] \times \R^d \rightarrow \P(A)$ such that
\begin{align*}
\hat{q}(t,X^1_t) = \E^{Q_1}\left[\left.\Lambda_t \right| X^1_t\right], \ Q_1-a.s., \ \text{a.e. } t \in [0,T].
\end{align*}
More precisely, we mean that for each bounded measurable function $\phi : [0,T] \times \R^d \times A \rightarrow \R$,
\begin{align}
\int_A\phi(t,X_t,a)\,\hat{q}(t,X^1_t)(da) = \E^{Q_1}\left[\left.\int_A\phi(t,X^1_t,a)\,\Lambda_t(da) \right| X^1_t\right], Q_1-a.s., \ \text{a.e. } t \in [0,T]. \label{pf:markovselection001}
\end{align}
To see this, define a probability measure $\eta$ on $[0,T] \times \R^d \times A$ by
\[
\eta(C) := \frac{1}{T}\E^{Q_1}\left[\int_0^T\int_A1_C(t,X^1_t,a)\Lambda_t(da)dt\right].
\]
We may then construct $\hat{q}$ by disintegration by writing $\eta(dt,dx,da) = \eta_{1,2}(dt,dx)[\hat{q}(t,x)](da)$, where $\eta_{1,2}$ denotes the $[0,T] \times \R^d$-marginal of $\eta$ and $\hat{q} : [0,T] \times \R^d \rightarrow \P(A)$ is measurable. Then, for each bounded measurable $h : [0,T] \times \R^d \rightarrow \R$, 
\begin{align*}
\E^{Q_1}\left[\int_0^Th(t,X^1_t)\int_A\phi(t,X_t,a)\hat{q}(t,X^1_t)(da)dt \right] &= T\int_{[0,T] \times \R^d}\!\!\!\!\!\!\!\!\!h(t,x)\int_A\phi(t,x,a)\,\hat{q}(t,x)(da)\eta_{1,2}(dt,dx) \\
	&= T\int_{[0,T] \times \R^d \times A}\!\!\!\!\!\!\!\!\!h(t,x)\phi(t,x,a)\eta(dt,dx,da) \\
	&= \E^{Q_1}\left[\int_0^Th(t,X^1_t)\int_A\phi(t,X^1_t,a)\,\Lambda_t(da)dt \right].
\end{align*}
This is enough to establish \eqref{pf:markovselection001}, thanks to \cite[Lemma 5.2]{brunickshreve-mimicking}. 

With $\hat{q}$ in hand, note that
\begin{align*}
\int_A\hat{q}(t,X^1_t)(da)b(t,X^1_t,\mu_t,a) &= \E^{Q_1}\left[\left.\int_A\Lambda_t(da)b(t,X^1_t,\mu_t,a)\right| X^1_t\right],
\end{align*}
and
\begin{align*}
\bar{\sigma}\bar{\sigma}^\top(t,X^1_t,\mu_t,\hat{q}(t,X^1_t)) &= \int_A\hat{q}(t,X^1_t)(da)\sigma\sigma^\top(t,X^1_t,\mu_t,a) \\
	&= \E^{Q_1}\left[\left.\int_A\Lambda_t(da)\sigma\sigma^\top(t,X^1_t,\mu_t,a)\right| X^1_t\right].
\end{align*}
The mimicking result of Brunick and Shreve \cite[Corollary 3.7]{brunickshreve-mimicking} tells us that there exists another filtered probability space $(\Omega^2,\F^2_t,Q_2)$ supporting a $\bar{m}$-dimensional $\F^2_t$-Wiener process $W^2$ and a $\R^d$-valued $\F^2_t$-adapted process $X^2$ such that
\begin{align}
dX^2_t &= \int_Ab(t,X^2_t,\mu_t,a)\hat{q}(t,X^2_t)(da)dt + \bar{\sigma}(t,X^2_t,\mu_t,\hat{q}(t,X^2_t))dW^2_t, \text{ and } \label{pf:markov01} \\
Q_2 \circ (X^2_t)^{-1} &= Q_1 \circ (X^1_t)^{-1} = P \circ X_t^{-1}, \text{ for all } t \in [0,T]. \label{pf:markov1}
\end{align}
It follows from It\^o's formula that $P_2 := Q_2 \circ (dt\hat{q}(t,X^2_t),X^2)^{-1}$ is in $\RC(\mu)$. Finally, compute
\begin{align*}
J(\mu,P_2) &= \E^{Q_2}\left[\int_0^T\int_Af(t,X^2_t,\mu_t,a)[\hat{q}(t,X^2_t)](da)dt + g(X^2_T,\mu_T) \right] \\
	&= \E^{Q_1}\left[\int_0^T\int_Af(t,X^1_t,\mu_t,a)[\hat{q}(t,X^1_t)](da)dt + g(X^1_T,\mu_T) \right] \\
	&= \E^{Q_1}\left[\int_0^T\int_Af(t,X^1_t,\mu_t,a)\Lambda_t(da)dt + g(X^1_T,\mu_T) \right] \\
	&= J(\mu,P).
\end{align*}
The second line follows from Fubini's theorem and \eqref{pf:markov1}. The third line follows from Fubini's theorem and the tower property of conditional expectations. This completes the proof of the first part of the theorem; set $P_0 = P_2$, and note that we have in fact proven (2) with \emph{equality}, not \emph{inequality}.

Now suppose assumption \ref{assumption:convex} holds. Then
\[
\left(b,\sigma\sigma^\top,f\right)(t,x,\mu_t,\hat{q}(t,x)) = \int_A\hat{q}(t,x)\left(b,\sigma\sigma^\top,f\right)(t,x,\mu_t,a) \in K(t,x,\mu_t),
\]
for each $(t,x) \in [0,T] \times \R^d$. As in \cite[Proposition 3.5]{haussmannlepeltier-existence}, $K(t,x,\mu_t)$ is a closed set for each $(t,x)$. By the measurable selection result of \cite[Theorem A.9]{haussmannlepeltier-existence} (or rather the slight extension of \cite[Lemma 3.1]{dufourstockbridge-existence}), there exist measurable functions $\hat{\alpha} : [0,T] \times \R^d \rightarrow A$ and $\hat{z} : [0,T] \times \R^d \rightarrow [0,\infty)$ such that
\begin{align}
\int_A\hat{q}(t,x)(da)\left(b,\sigma\sigma^\top,f\right)(t,x,\mu_t,a) = \left(b,\sigma\sigma^\top,f\right)(t,x,\mu_t,\hat{\alpha}(t,x)) - \left(0,0,\hat{z}(t,x)\right), \label{pf:markov2}
\end{align}
for all $(t,x) \in [0,T] \times \R^d$. In particular,
\begin{align}
b(t,x,\mu_t,\hat{\alpha}(t,x)) &= \int_A\hat{q}(t,x)(da)b(t,x,\mu_t,a), \text{ and } \nonumber \\
\sigma\sigma^\top(t,x,\mu_t,\hat{\alpha}(t,x)) &= \int_A\hat{q}(t,x)(da)\sigma\sigma^\top(t,x,\mu_t,a) \nonumber \\
	&= \bar{\sigma}\bar{\sigma}^\top(t,x,\mu_t,\hat{q}(t,x)) \label{pf:markov3}
\end{align}
Now define
\[
P_0 := Q_2 \circ (dt\delta_{\hat{\alpha}(t,X^2_t)}(da),X^2)^{-1}.
\]
Using the equality \eqref{pf:markov3} and It\^{o}'s formula, we conclude that $P_0$ is in $\RC(\mu)$. Intuitively, we are exploiting here the fact that the law of the solution of an SDE does not depend on the choice of square root of the volatility matrix.
Finally, 
\begin{align*}
J(\mu,P_0) &= \E^{Q_2}\left[\int_0^Tf(t,X^2_t,\mu_t,\hat{\alpha}(t,X^2_t))dt + g(X^2_T,\mu_T) \right] \\
	&\ge \E^{Q_2}\left[\int_0^T\int_Af(t,X^2_t,\mu_t,a)\hat{q}(t,X^2_t)(da)dt + g(X^2_T,\mu_T) \right] \\
	&= \E^{Q_1}\left[\int_0^T\int_Af(t,X^1_t,\mu_t,a)\hat{q}(t,X^1_t)(da)dt + g(X^1_T,\mu_T) \right] \\
	&= \E^{Q_1}\left[\int_0^T\int_Af(t,X^1_t,\mu_t,a)\Lambda_t(da)dt + g(X^1_T,\mu_T) \right] \\
	&= J(\mu,P).
\end{align*}
The second line follows from \eqref{pf:markov2}. The third line comes from Fubini's theorem and $Q_2 \circ (X^2_t)^{-1} = Q_1 \circ (X^1_t)^{-1}$, $t \in [0,T]$. The fourth line follows from Fubini's theorem and the tower property of conditional expectations. The last step is just \eqref{pf:markov0}.
\end{proof}

\begin{remark} \label{re:markovselection}
It should be noted that the control produced by Theorem \ref{th:markovselection} is called \emph{Markovian} because of its form $\hat{\alpha}(t,X_t)$, but it does not necessarily render the state process $X$ a Markov process. Although the dynamics appear to be Markovian, the process $X$ is a solution of a potentially ill-posed martingale problem, and it is well-known (see \cite[Chapter 12]{stroockvaradhanbook}) that uniqueness in law is required to guarantee the solution is Markovian. If the volatility $\sigma$ is uncontrolled and uniformly nondegenerate, then the martingale problem is indeed well-posed, and $X$ is a strong Markov process.
\end{remark}

\begin{remark} \label{re:markovselection2}
It is clear from the proof that the full force of assumption \ref{assumption:A} is not needed for Theorem \ref{th:markovselection}. Notably, the assumption $p' > p$ is not needed. The coefficients $(b,\sigma,f)$ should be continuous in $a$ to ensure that the set $K(t,x,\mu)$ is closed, but continuity in $(x,\mu)$ is unnecessary.
\end{remark}

\section{Bounded coefficients} \label{se:boundedcoefficients}
In this section, Theorem \ref{th:relaxedexistence} is proven in the case that the coefficients are bounded and the control space compact. The general case is proven in Section \ref{se:unboundedcoefficients} by a limiting argument. Consider the following assumption:

\begin{assumption}{\textbf{(B)}} \label{assumption:B}
The functions $b$, $\sigma$ are bounded, and the control space $A$ is compact.
\end{assumption}

\begin{theorem} \label{th:boundedexistence}
Under assumptions \ref{assumption:A} and \ref{assumption:B}, there exists a relaxed MFG solution.
\end{theorem}

\begin{remark}
In fact, under assumptions \ref{assumption:A} and \ref{assumption:B}, we may take $p'=p=0$ in assumption \ref{assumption:A}, and Theorem \ref{th:boundedexistence} is true with an even simpler proof.
\end{remark}

The proof of Theorem \ref{th:boundedexistence} is broken up into several lemmas. First, we state a version of a standard estimate which will be useful in later sections as well. Recall here the notational convention of \eqref{def:measureconvention}.

\begin{lemma} \label{le:stateestimate}
Assume \ref{assumption:A} holds, and fix $\gamma \in [p,p']$. Then there exists a constant $c_4 > 0$, depending only on $\gamma$, $|\lambda|^{p'}$, $T$, and the constant $c_1$ of (A.2) such that for any $\mu \in \P^p(\C^d)$ and $P \in \RC[b,\sigma,A](\mu)$ we have
\begin{align*}
\E^P\|X\|_T^\gamma &\le c_4\left(1 + \|\mu\|_T^\gamma + \E^P\int_0^T|\Lambda_t|^\gamma dt\right).
\end{align*}
In particular, $P \in \P^p(\Omega)$. Moreover, if $P \circ X^{-1} = \mu$, then we have
\begin{align*}
\|\mu\|_T^\gamma = \E^P\|X\|_T^\gamma &\le c_4\left(1 + \E^P\int_0^T|\Lambda_t|^\gamma dt\right).
\end{align*}
\end{lemma}
\begin{proof}
There is a constant $C > 0$ (which will change from line to line) such that
\begin{align*}
|X_t|^\gamma \le &C|X_0|^\gamma + C\int_0^tds\int_A\Lambda_s(da)|b(s,X_s,\mu_s,a)|^\gamma \\
	&+ C\left|\int_0^t\int_A\sigma(s,X_s,\mu_s,a)N(da,ds)\right|^\gamma.
\end{align*}
The Burkholder-Davis-Gundy inequality yields
\begin{align*}
\E^P\|X\|_t^\gamma \le &C\E^P\left[|X_0|^\gamma + \int_0^tds\int_A\Lambda_s(da)\sup_{0 \le u \le s}|b(u,X_u,\mu_u,a)|^\gamma \right. \\
	&\quad\quad+ \left.\left(\int_0^tds\int_A\Lambda_s(da)\sup_{0 \le u \le s}|\sigma(s,X_s,\mu_s,a)|^2\right)^{\gamma/2}\right] \\
	\le &C\E^P\left[|X_0|^\gamma + \int_0^tds\int_A\Lambda_s(da)c_1^\gamma(1 + \|X\|_s^\gamma + \|\mu\|_s^\gamma + |a|^\gamma)  \right. \\
	&\quad\quad+ \left.\left(\int_0^tds\int_A\Lambda_s(da)c_1\left(1 + \|X\|_s^{p_\sigma} + \left(\int_{\C^d}\|z\|_s^p\mu(dz)\right)^{p_\sigma/p} + |a|^{p_\sigma}\right)\right)^{\gamma/2}\right] \\
	&\le C\E^P\left[1 + |X_0|^\gamma + \int_0^tds\int_A\Lambda_s(da)(1 + \|X\|_s^\gamma + \|\mu\|_s^\gamma + |a|^\gamma)\right]
\end{align*}
We used Jensen's inequality for the second line to get $\left(\int_{\C^d}\|z\|_s^p\mu(dz)\right)^{\gamma/p} \le \|\mu\|_s^\gamma$. If $\gamma \ge 2$, the last line follows from Jensen's inequality and the inequality $|x|^{p_\sigma \gamma/2} \le 1 + |x|^\gamma$, which holds since $p_\sigma \le 2$. If $\gamma / 2 \le 1$, the last line follows from the inequality $|x|^{\gamma / 2} \le 1 + |x|$ followed by $|x|^{p_\sigma} \le 1 + |x|^\gamma$, which holds since $\gamma \ge p_\sigma$. The first claim follows now from Gronwall's inequality. If $P \circ X^{-1} = \mu$, then the above becomes
\begin{align*}
\|\mu\|_t^\gamma = \E^P\|X\|_t^\gamma \le C\E^P\left[|X|_0^\gamma + \int_0^t\left(1 + 2\|\mu\|_s^\gamma + |\Lambda_t|^\gamma\right)ds\right].
\end{align*}
The second claim now also follows from Gronwall's inequality.
\end{proof}

The proof of Theorem \ref{th:boundedexistence} is an application of the Kakutani-Fan-Glicksberg fixed point theorem. For background on set-valued analysis the reader is referred to \cite[Chapter 17]{aliprantisborder}. For this paragraph, fix two metric spaces $E$ and $F$. A set valued function $h : E \rightarrow 2^F$ is \emph{lower hemicontinuous} if, whenever $x_n \rightarrow x$ in $E$ and $y \in h(x)$, there exists $y_{n_k} \in h(x_{n_k})$ such that $y_{n_k} \rightarrow y$. If $h(x)$ is closed for each $x \in E$ then $h$ is called \emph{upper hemicontinuous} if, whenever $x_n \rightarrow x$ in $E$ and $y_n \in h(x_n)$ for each $n$, the sequence $(y_n)$ has a limit point in $h(x)$. We say $h$ is \emph{continuous} if it is both upper hemicontinuous and lower hemicontinuous. If $F$ is compact, then the graph $\{(x,y) : x \in E, \ y \in h(x)\}$ of $h$ is closed if and only if $h(x)$ is closed for each $x \in E$ and $h$ is upper hemicontinuous.

\begin{lemma} \label{le:rccontinuous}
Under assumptions \ref{assumption:A} and \ref{assumption:B}, the range $\RC(\P^p(\C^d))) := \{P \in \RC(\mu) : \mu \in \P^p(\C^d)\}$ is relatively compact in $\P^p(\Omega)$, and the set-valued function $\RC$ is continuous.
\end{lemma}
\begin{proof}
When $A$ is compact, so is $\V=\V[A]$, and the topology of $\P^p(\V)$ is that of weak convergence. Thus $\{P \circ \Lambda^{-1} : P \in \RC(\P^p(\C^d))\}$ is relatively compact in $\P^p(\V)$. From Proposition \ref{pr:itocompact} and boundedness of $b$ and $\sigma$ it follows that $\{P \circ X^{-1} : P \in \RC(\P^p(\C^d))\}$ is relatively compact in $\P^p(\C^d)$. Thus $\RC(\P^p(\C^d))$ is relatively compact in $\P^p(\Omega)$, by Lemma \ref{le:productrelcompactness}.

To show $\RC$ is upper hemicontinuous, it suffices show its graph is closed, since its range is relatively compact. Let $\mu^n \rightarrow \mu$ in $\P^p(\C^d)$ and $P^n \rightarrow P$ in $\P^p(\Omega)$ with $P^n \in \RC(\mu^n)$. Clearly $P \circ X_0^{-1} = \lim_nP^n \circ X_0^{-1} = \lambda$. Now fix $s < t$, $\phi \in C^\infty_0(\R^d)$, and a bounded, continuous, and $\F_s$-measurable function $h : \Omega \rightarrow \R$. Note that $(\mu,q,x) \mapsto M^{\mu,\phi}_t(q,x)$ is bounded and continuous (apply Corollary \ref{co:V}(2) with $p=0$). Since $M^{\mu^n,\phi}_t$ is a $P^n$-martingale for each $n$,
\[
\E^P\left[(M^{\mu,\phi}_t-M^{\mu,\phi}_s)h\right] = \lim_{n\rightarrow\infty}\E^{P^n}\left[(M^{\mu^n,\phi}_t-M^{\mu^n,\phi}_s)h\right] = 0.
\]
Hence $M^{\mu,\phi}_t$ is a $P$-martingale, and so $P \in \RC(\mu)$.

To show $\RC$ is lower hemicontiuous, let $\mu^n \rightarrow \mu$ and $P \in \RC(\mu)$. By Proposition \ref{pr:sderepresentation}, there exists a filtered probability space $(\Omega',\F'_t,P')$ supporting a $d$-dimensional $\F'_t$-adapted process $X$ as well as $m$ orthogonal $\F'_t$-martingale measures $N = (N^1,\ldots,N^m)$ on $A \times [0,T]$ with intensity $\Lambda_t(da)dt$, such that $P' \circ (\Lambda,X)^{-1} = P$ and the state equation \eqref{def:SDE} holds on $\Omega'$. The Lipschitz assumption (A.2) ensures that for each $n$ we may strongly solve the SDE
\[
dX^n_t = \int_Ab(t,X^n_t,\mu^n_t,a)\Lambda_t(da)dt + \int_A\sigma(t,X^n_t,\mu^n_t,a)N(da,dt), \ X^n_0 = X_0.
\]
Let $\gamma \ge 2$. A standard estimate using the Lipschitz assumption and the Burkholder-Davis-Gundy inequality yields a constant $C > 0$ independent of $n$ (which may change from line to line) such that
\begin{align*}
\E^{P'}\|X^n - X\|_t^\gamma &\le C\E^{P'}\int_0^t\int_A|b(s,X^n_s,\mu^n_s,a) - b(s,X_s,\mu_s,a)|^\gamma\Lambda_s(da)ds \\
	&\quad + C\E^{P'}\int_0^t\int_A\left|\sigma(s,X^n_s,\mu^n_s,a) - \sigma(s,X_s,\mu_s,a)\right|^\gamma\Lambda_s(da)ds \\
	&\le C\int_0^t\|X^n-X\|^\gamma_sds + C\E^{P'}\int_0^t\int_A|b(s,X_s,\mu^n_s,a) - b(s,X_s,\mu_s,a)|^\gamma\Lambda_s(da)ds \\
	&\quad + C\E^{P'}\int_0^t\int_A|\sigma(s,X_s,\mu^n_s,a) - \sigma(s,X_s,\mu_s,a)|^\gamma\Lambda_s(da)ds.
\end{align*}
Since $b$ and $\sigma$ are bounded and continuous in $\mu$, Gronwall's inequality and the dominated convergence theorem yield $\E^{P'}\|X^n - X\|_T^\gamma \rightarrow 0$. Let $P^n := P' \circ (\Lambda,X^n)^{-1}$, and check using It\^o's formula that $P^n \in \RC(\mu^n)$. Choosing $\gamma \ge p$ implies $P^n \rightarrow P$ in $\P^p(\Omega)$, and the proof is complete.
\end{proof}

\begin{lemma} \label{le:Jcontinuous}
Suppose assumption \ref{assumption:A} holds. Then $J$ is upper semicontinuous. If also \ref{assumption:B} holds, then $J$ is continuous.
\end{lemma}
\begin{proof}
It follows from Corollary \ref{co:V} and the upper bounds of $f$ and $g$ of assumption (A.5) that $\P^p(\C^d) \times \V \times \C^d \ni (\mu,q,x) \mapsto \Gamma^\mu(q,x)$ is upper semicontinuous. Hence, $J$ is upper semicontinuous. If $A$ is compact, then $\Gamma$ is continuous by Corollary \ref{co:V}, and so $J$ is continuous.
\end{proof}

\begin{proof}[Proof of Theorem \ref{th:boundedexistence}]
Since $\RC$ is continuous and has nonempty compact values (Lemma \ref{le:rccontinuous}), and since $J$ is continuous (Lemma \ref{le:Jcontinuous}), it follows from a famous result of Berge \cite[Theorem 17.31]{aliprantisborder} that $\RC^*$ is upper hemicontinuous. It is clear that $\RC(\mu)$ is convex for each $\mu$, and it follows from linearity of $P \mapsto J(\mu,P)$ that $\RC^*(\mu)$ is convex for each $\mu$. The map $\P^p(\Omega) \ni P \mapsto P \circ X^{-1} \in \P^p(\C^d)$ is linear and continuous, and it follows that the set-valued map
\[
\P^p(\C^d) \ni \mu \mapsto F(\mu) := \left\{ P \circ X^{-1} : P \in \RC^*(\mu)\right\} \subset \P^p(\C^d)
\]
is upper hemicontinuous and has nonempty compact convex values. To apply a fixed point theorem, we must place the range $F(\P^p(\C^d))$ inside of a convex compact subset of a nice topological vector space. To this end, define
\[
M := \sup\left\{\|\mu\|^{p'}_T : \mu \in F(\P^p(\C^d))\right\} < \infty.
\]
By assumption, $b$ and $\sigma$ are bounded, so for each $\phi \in C^\infty_0(\R^d)$ we may find $C_\phi > 0$ such that
\begin{align*}
|L\phi(t,x,\mu,a)| \le C_\phi,
\end{align*}
for all $(t,x,\mu,a)$. Moreover, $C_\phi$ depends only on $D\phi$ and $D^2\phi$. Let $\Q$ denote the set of probability measures $P$ on $\C^d$ satisfying the following: 
\begin{enumerate}
\item $P \circ X_0^{-1} = \lambda$,
\item $\E^P\|X\|^{p'}_T \le M$,
\item For each nonnegative $\phi \in C^\infty_0(\R^d)$, the process $\phi(X_t) + C_\phi t$ is a $P$-submartingale.
\end{enumerate}
It is clear both that $\Q$ is convex and that $F(\P^p(\C^d))$ is contained in $\Q$. It follows from \cite[Theorem 1.4.6]{stroockvaradhanbook} that $\Q$ is tight, and the $p'$-moment bound (2) ensures that it is relatively compact in $\P^p(\C^d)$. In fact, it is straightforward to check that $\Q$ is closed in $\P^p(\C^d)$, and thus it is compact.

Now note that $\Q$ is a subset of the space $\M(\C^d)$ of bounded signed measures on $\C^d$. When endowed with the topology $\tau_w$ of weak convergence, i.e. the topology $\tau_w = \sigma(\M(\C^d),C_b(\C^d))$ induced by bounded continuous functions, $\M(\C^d)$ is a locally convex Hausdorff space. Since $\Q$ is relatively compact in $\P^p(\C^d)$, the $p$-Wasserstein metric $d_{\C^d,p}$ on $\P^p(\C^d)$ and the topology $\tau_w$ on $\M(\C^d)$ both induce the same topology on $\Q$. Hence, $\Q$ is $\tau_w$-compact. The set-valued function $F$ maps $\Q$ into itself, it is upper hemicontinuous with respect to $\tau_w$ (equivalently, its graph is closed), and its values are nonempty, compact, and convex. Existence of a fixed point now follows from the Kakutani-Fan-Glicksberg theorem; see \cite[Theorem 1]{fan-fixedpoint1952} or \cite[Corollary 17.55]{aliprantisborder}.
\end{proof}

\begin{remark} \label{re:objectivegeneralization2}
If one is not interested in Markovian solutions, it is evident from the proofs of this section that a relaxed existence result holds with much more general objective structures, as indicated in Remark \ref{re:objectivegeneralization0}. In particular, we only used the fact that $J : \P^p(\C^d) \times \P^p(\Omega) \rightarrow \R$ is continuous and concave.
\end{remark}

\section{Unbounded coefficients} \label{se:unboundedcoefficients}
This section is devoted to the proof of Theorem \ref{th:relaxedexistence}, without assuming that $b$, $\sigma$, and $A$ are bounded. Assume throughout this section that assumption \ref{assumption:A} holds. Naturally, the idea is to approximate the data $(b,\sigma,A)$ with truncated versions which satisfy \ref{assumption:B}. Let $b_n$ and $\sigma_n$ denote the (pointwise) projections of $b$ and $\sigma$ into the ball centered at the origin with radius $n$ in $\R^d$ and $\R^{d \times m}$, respectively. Let $A_n$ denote the intersection of $A$ with the ball centered at the origin with radius $r_n$, where
\begin{align}
r_n := [n/(2c_1)]^{1/2}. \label{def:rndef}
\end{align}
(Recall that the constant $c_1$ comes from assumption (A.2).) For sufficiently large $n_0$, $A_n$ is nonempty and compact for all $n \ge n_0$, and thus we will always assume $n \ge n_0$ in what follows. Note that the truncated data $(b_n,\sigma_n,f,g,A_n)$ satisfy \ref{assumption:B} as well as \ref{assumption:A}. Moreover, (A.2) and (A.3) hold \emph{with the same constants} $c_1,c_2,c_3$.

By Theorem \ref{th:boundedexistence} there exists for each $n$ a corresponding MFG solution, which is technically a measure on $\Omega[A_n] = \V[A_n] \times \C^d$ but may naturally be viewed as a measure on $\Omega$, since $A_n \subset A$. To clarify: Since $A_n \subset A$ there is a natural embedding $\V[A_n] \hookrightarrow \V[A]$. Define $\RC_n(\mu)$ to be the set of $P \in \P(\Omega[A])$ satisfying the following:
\begin{enumerate}
\item $P(\Lambda([0,T] \times A^c_n)=0)=1$.
\item $P \circ X_0^{-1} = \lambda$.
\item $M^{\mu,\phi}[b_n,\sigma_n,A]$ is a $P$-martingale for each $\phi \in C^\infty_0(\R^d)$.
\end{enumerate}
Define
\[
\RC^*_n(\mu) := \arg\max_{P' \in \RC_n(\mu)}J^{\mu}[f,g,A](P)
\]
Then it is clear that $\RC_n(\mu)$ (resp. $\RC^*_n(\mu)$) is exactly the image of the set $\RC[b_n,\sigma_n,f,g,A_n](\mu)$ (resp. $\RC^*[b_n,\sigma_n,f,g,A_n](\mu)$) under the natural embedding $\P(\Omega[A_n]) \hookrightarrow \P(\Omega[A])$. Henceforth, we identify these sets. By Theorem \ref{th:boundedexistence}, there exist corresponding MFG solutions which may be interpreted as $\mu^n \in \P^p(\C^d)$ and $P_n \in \RC^*_n(\mu^n)$ with $\mu^n = P_n \circ X^{-1}$.

\subsection{Relative compactness of the approximations}
The strategy of the proof is to show that $P_n$ are relatively compact and then characterize the limit points as MFG solutions for the original data $(b,\sigma,f,g,A,\lambda)$. The following Lemma \ref{le:tight+moments} makes crucial use of the upper bound on $f$ of assumption (A.3) along with the assumption $p' > p$, in order to establish some uniform integrability of the controls.

\begin{lemma} \label{le:tight+moments}
The measures $P_n$ are relatively compact in $\P^p(\Omega[A])$. Moreover,
\begin{align}
\sup_n\E^{P_n}\int_0^T|\Lambda_t|^{p'}dt &< \infty \label{pf:unbounded3} \\
\sup_n\E^{P_n}\|X\|_T^{p'} = \sup_n\|\mu^n\|_T^{p'} &< \infty. \label{pf:unbounded4}
\end{align}
\end{lemma}
\begin{proof}
Noting that the coefficients $(b_n,\sigma_n)$ satisfy \ref{assumption:A} with the same constants (independent of $n$), the second conclusion of Lemma \ref{le:stateestimate} implies
\begin{align}
\|\mu^n\|_T^p = \E^{P_n}\|X\|^p_T \le c_4\left(1 + \E^{P_n}\int_0^T|\Lambda_t|^pdt\right). \label{pf:unbounded1}
\end{align}
Fix $a_0 \in A_{n_0}$. For $n \ge n_0$, let $Q_n$ denote the unique element of $\RC_n(\mu^n)$ satisfying $Q_n(\Lambda_t = \delta_{a_0} \text{ for a.e. } t) = 1$. That is, $Q_n$ is the law of the solution of the state equation arising from the constant control equal to $a_0$. The first part of Lemma \ref{le:stateestimate} implies
\begin{align}
\E^{Q_n}\|X\|^{p}_T &\le c_4\left(1 + \|\mu^n\|_T^{p} + T|a_0|^{p}\right) \le C_0\left(1 + \E^{P_n}\int_0^T|\Lambda_t|^{p}dt\right), \label{pf:unbounded2}
\end{align}
where the constant $C_0 > 0$ depends only on $c_4$, $T$, $p$, and $a_0$. Use the optimality of $P_n$, the lower bounds on $f$ and $g$, and then \eqref{pf:unbounded1} and \eqref{pf:unbounded2} to get
\begin{align}
J(\mu^n,P_n) &\ge J(\mu^n,Q_n) \ge -c_2(T+1)\left(1 + \E^{Q_n}\|X\|_T^{p} + \|\mu^n\|_T^{p} + |a_0|^{p'}\right) \nonumber \\
	&\ge -C_1\left(1 + \E^{P_n}\int_0^T|\Lambda_t|^pdt\right), \label{pf:unbounded2-1}
\end{align}
where $C_1 > 0$ depends only on $c_2$, $c_4$, $T$, $p$, $p'$, and $a_0$. On the other hand, we may use the upper bounds on $f$ and $g$ along with \eqref{pf:unbounded1} to get
\begin{align}
J(\mu^n,P_n) &\le c_2(T+1)\left(1 + \E^{P_n}\|X\|_T^p + \|\mu^n\|^p_T\right) - c_3\E^{P_n}\int_0^T|\Lambda_t|^{p'}dt \nonumber \\
	&\le C_2\left(1 + \E^{P_n}\int_0^T|\Lambda_t|^pdt\right) - c_3\E^{P_n}\int_0^T|\Lambda_t|^{p'}dt, \label{pf:unbounded2-2}
\end{align}
where $C_2 > 0$ depends only on $c_2$, $c_3$, $c_4$, $T$, $p$, and $a_0$. Combining \eqref{pf:unbounded2-1} and \eqref{pf:unbounded2-2} and rearranging, we find two constants, $\kappa_1 \in \R$ and $\kappa_2 > 0$, such that
\[
\E^{P_n}\int_0^T(|\Lambda_t|^{p'} + \kappa_1|\Lambda_t|^p)dt \le \kappa_2.
\]
(Note that $\E^{P_n}\int_0^T|\Lambda_t|^pdt < \infty$ for each $n$.) Crucially, these constants are independent of $n$. Since $p' > p$, it holds for all sufficiently large $x$ that $x^{p'} + \kappa_1 x^p \ge x^{p'}/2$, and \eqref{pf:unbounded3} follows. Combined with the second conclusion of Lemma \ref{le:stateestimate}, this implies \eqref{pf:unbounded4}. Finally, relative compactness of $P_n$ is proven by an application of Aldous' criterion, detailed in Proposition \ref{pr:itocompact}.
\end{proof}

\subsection{Limiting state process dynamics}
Now that we know $P_n$ are relatively compact, we may fix $P \in \P^p(\Omega[A])$ and a subsequence $n_k$ such that $P_{n_k} \rightarrow P$ in $\P^p(\Omega[A])$. Define $\mu := P \circ X^{-1}$, and note that $\mu^{n_k} \rightarrow \mu$ in $\P^p(\C^d)$.

\begin{lemma} \label{le:limitpoint}
The limit point $P$ satisfies $P \in \RC[b,\sigma,A](\mu)$, $\mu = P \circ X^{-1}$, and also
\[
\E^P\int_0^T|\Lambda_t|^{p'}dt < \infty.
\]
\end{lemma}
\begin{proof}
It is immediate that $\mu = \lim_k\mu^{n_k} = \lim_kP_{n_k} \circ X^{-1} = P \circ X^{-1}$, and in particular $P \circ X_0^{-1} = \lambda$. Fatou's lemma and \eqref{pf:unbounded3} imply
\[
\E^P\int_0^T|\Lambda_t|^{p'}dt \le \liminf_{k\rightarrow\infty}\E^{P_{n_k}}\int_0^T|\Lambda_t|^{p'}dt < \infty.
\]
We must only prove $P \in \RC[b,\sigma,A](\mu)$. Fix $\phi \in C^\infty_0(\R^d)$, and note that $M^{\mu^n,\phi}_t[b_n,\sigma_n,A_n]$ is a $P_n$ martingale for each $n$. We must show that $M^{\mu,\phi}_t[b,\sigma,A]$ is a $P$-martingale.

Note that $M^{\mu^n,\phi}_t[b_n,\sigma_n,A]$ may be identified with $M^{\mu^n,\phi}_t[b_n,\sigma_n,A_n]$, since $P_n$-almost surely $\Lambda$ is concentrated on $[0,T] \times A_n$. Letting $L_n$ denote the generator associated to $(b_n,\sigma_n)$, we have
\begin{align}
M^{\mu^n,\phi}_t&[b_n,\sigma_n,A](q,x) - M^{\mu^n,\phi}_t[b,\sigma,A](q,x) \nonumber \\
	&= \int_0^tds\int_A\Lambda_s(da)\left(L_n\phi(s,x_s,\mu^n_s,a) - L\phi(s,x_s,\mu^n_s,a)\right) \nonumber \\
	&= \int_0^tds\int_A\Lambda_s(da)\left(b_n(s,x_s,\mu^n_s,a) - b(s,x_s,\mu^n_s,a)\right)^\top D\phi(x_s) + \nonumber \\
	&\quad\quad\quad\quad + \frac{1}{2}\text{Tr}\left[\left(\sigma_n\sigma_n^\top(s,x_s,\mu^n_s,a) - \sigma\sigma^\top(s,x_s,\mu^n_s,a)\right)D^2\nabla\phi(x_s)\right]. \label{pf:mlimit}
\end{align}
By construction, $b_n(s,x_s,\mu^n_s,a) \neq b(s,x_s,\mu^n_s,a)$ implies $|b(s,x_s,\mu^n_s,a)| > n$, which by assumption (A.2) implies 
\begin{align}
n < c_1\left(1 + |x_s| + \left(\int_{\R^d}|z|^p\mu^n_s(dz)\right)^{1/p} + |a|\right). \label{pf:mlimit00}
\end{align}
Moreover, $|b_n(s,x_s,\mu^n_s,a) - b(s,x_s,\mu^n_s,a)|$ is bounded above by twice the right-hand side of \eqref{pf:mlimit00}.
For $\gamma \in (0,p']$, denote
\[
Z_\gamma := 1 + \|X\|_T^\gamma + \left(\sup_n\int_{\C^d}\|z\|_T^p\mu^n(dz)\right)^{\gamma/p},
\]
noting that the supremum is finite by Lemma \ref{le:tight+moments}. Let $C > 0$ bound the first two derivatives of $\phi$. Because of the definition \eqref{def:rndef} of $r_n$, for $n \ge 2c_1$ and $\gamma \in [0,2]$ we have
\begin{align*}
\Lambda\{(t,a) : 2c_1|a|^\gamma > n\} \le \Lambda\{(t,a) : 2c_1|a|^2 > n\} = 0, \ P_n-a.s. 
\end{align*}
Hence
\begin{align*}
\int_0^tds & \int_A\Lambda_s(da)\left|\left(b_n(s,X_s,\mu^n_s,a) - b(s,X_s,\mu^n_s,a)\right)^\top D\phi(X_s)\right| \\
	&\le C\int_0^tds\int_A\Lambda_s(da)2c_1(Z_1 + |a|)1_{\{c_1(Z_1 + |a|) > n\}} \\
	&\le 2Cc_1\int_0^tds\int_A\Lambda_s(da)\left(Z_1 + |a|\right)\left(1_{\{2c_1Z_1 > n\}} + 1_{\{2c_1|a| > n\}}\right) \\
	&\le 2Cc_1\left(tZ_1 + \int_0^t|\Lambda_s|ds\right)1_{\{2c_1Z_1 > n\}}, \ P_n-a.s.
\end{align*}
We have a similar bound for the $\sigma_n\sigma_n^\top-\sigma\sigma^\top$ term:
\begin{align*}
\int_0^tds & \int_A\Lambda_s(da)\left|\text{Tr}\left[\left(\sigma_n\sigma_n^\top(s,X_s,\mu^n_s,a) - \sigma\sigma^\top(s,X_s,\mu^n_s,a)\right)D^2\nabla\phi(x_s)\right]\right| \\
	&\le C\int_0^tds\int_A\Lambda_s(da)2c_1(Z_{p_\sigma} + |a|^{p_\sigma})1_{\{c_1(Z_{p_\sigma} + |a|^{p_\sigma}) > n\}} \\
	&\le 2Cc_1\left(tZ_{p_\sigma} + \int_0^t|\Lambda_s|^{p_\sigma}ds\right)1_{\{2c_1Z_{p_\sigma} > n\}}, \ P_n-a.s.
\end{align*}
Note that \eqref{pf:unbounded4} implies $\sup_n\|\mu^n\|_T^p < \infty$. Returning to \eqref{pf:mlimit}, it holds $P_n$-a.s. that
\begin{align*}
&\left|M^{\mu^n,\phi}_t[b_n,\sigma_n,A] - M^{\mu^n,\phi}_t[b,\sigma,A]\right| \\
&\quad\quad \le 2Cc_1\left[\left(TZ_1 + \int_0^T|\Lambda_s|ds\right)1_{\{2c_1Z_1 > n\}} +  \left(TZ_{p_\sigma} + \int_0^T|\Lambda_s|^{p_\sigma}ds\right)1_{\{2c_1Z_{p_\sigma} > n\}}\right] 
%&\quad\quad \le 2Cc_1\left(T(Z_1 + Z_{p_\sigma}) + \int_0^T\left(|\Lambda_s| + |\Lambda_s|^{p_\sigma}\right)ds\right)\left(1_{\{2c_1Z_1 > n\}} + 1_{\{2c_1Z_{p_\sigma} > n\}}\right) 
\end{align*}
for all $t \in [0,T]$. Since $1 \vee p_\sigma \le p < p'$ by assumption (A.5), and since Lemma \ref{le:tight+moments} yields
\[
\sup_n\E^{P_n}\left[\|X\|^{p'}_T + \int_0^T|\Lambda_t|^{p'}dt\right] < \infty,
\]
we have 
\begin{align}
\lim_{n\rightarrow\infty}\E^{P_n}\left|M^{\mu^n,\phi}_t[b_n,\sigma_n,A] - M^{\mu^n,\phi}_t[b,\sigma,A]\right| = 0. \label{pf:mlimit2}
\end{align}
On the other hand, the map
\[
\P^p(\C^d) \times \Omega[A] \ni (\nu,q,x) \mapsto M^{\nu,\phi}_t[b,\sigma,A](q,x) \in \R
\]
is jointly continuous for each $t$, by Corollary \ref{co:V}(2). Fix $s < t$ and a bounded, continuous, and $\F_s$-measurable $h : \Omega \rightarrow \R$. Then, since $P_{n_k} \rightarrow P$ in $\P^p(\Omega[A])$, and since $M^{\mu,\phi}$ grows with order $1 \vee p_\sigma \le p$, we have (by Proposition \ref{pr:wasserstein})
\begin{align}
\lim_{k\rightarrow\infty}\E^{P_{n_k}}&\left[\left(M^{\mu^{n_k},\phi}_t[b,\sigma,A] - M^{\mu^{n_k},\phi}_s[b,\sigma,A]\right)h \right] \nonumber \\
	&= \E^{P}\left[\left(M^{\mu,\phi}_t[b,\sigma,A] - M^{\mu,\phi}_s[b,\sigma,A]\right)h \right]. \label{pf:mlimit3}
\end{align}
Since $M^{\mu^n,\phi}_t[b_n,\sigma_n,A]$ is a $P_n$-martingale, combining \eqref{pf:mlimit2} and \eqref{pf:mlimit3} yields
\begin{align*}
0 &= \lim_{k\rightarrow\infty}\E^{P_{n_k}}\left[\left(M^{\mu^{n_k},\phi}_t[b_{n_k},\sigma_{n_k},A] - M^{\mu^{n_k},\phi}_s[b_{n_k},\sigma_{n_k},A]\right)h \right] \\
	&= \E^{P}\left[\left(M^{\mu,\phi}_t[b,\sigma,A] - M^{\mu,\phi}_s[b,\sigma,A]\right)h \right].
\end{align*}
Hence $M^{\mu,\phi}_t[b,\sigma,A]$ is a $P$-martingale, and the proof is complete.
\end{proof}

\subsection{Optimality of the limiting control}
It remains to show that the limit point $P$ is optimal, or $P \in \RC^*[b,\sigma,f,g,A](\mu)$. The crucial tool is the following lemma.

\begin{lemma} \label{le:limp'}
For each $P' \in \RC[b,\sigma,A](\mu)$ such that $J[f,g,A](\mu,P') > -\infty$, there exists $P'_n \in \RC_n(\mu^n)$ such that
\begin{align}
J[f,g,A](\mu,P') &= \lim_{k\rightarrow\infty}J[f_{n_k},g_{n_k},A_{n_k}](\mu^{n_k},P'_{n_k}). \label{pf:optimal3}
\end{align}
\end{lemma}
\begin{proof}
First, the upper bounds of $f$ and $g$ imply
\begin{align*}
J[f,g,A](\mu,P') &\le c_2(T+1)\left(1 + \E^{P'}\|X\|_T^p + \|\mu\|^p_T\right) - c_3\E^{P'}\int_0^Tdt|\Lambda_t|^{p'}.
\end{align*}
Since $\|\mu\|^p_T < \infty$ and $\E^{P'}\|X\|_T^p < \infty$, the assumption $J[f,g,A](\mu,P') > -\infty$ implies
\begin{align}
\E^{P'}\int_0^Tdt|\Lambda_t|^{p'} < \infty. \label{pf:lambdaintegrable2}
\end{align}
By Proposition \ref{pr:sderepresentation}, we may find a filtered probability space $(\Omega',\F'_t,Q')$ supporting a $d$-dimensional $\F'_t$-adapted process $X$ as well as $m$ orthogonal $\F'_t$-martingale measures $N = (N^1,\ldots,N^m)$ on $A \times [0,T]$ with intensity $\Lambda_t(da)dt$, such that $Q' \circ (\Lambda,X)^{-1} = P'$ and the state equation \eqref{def:SDE} holds. Find a measurable map $\iota_n : A \rightarrow A$ such that $\iota_n(A) \subset A_n$ and $\iota_n(a) = a$ for all $a \in A_n$, so that $\iota_n$ converges pointwise to the identity. Let $X^n$ denote the unique strong solution of
\[
dX^n_t = \int_Ab_n(t,X^n_t,\mu^n_t,\iota_n(a))\Lambda_t(da)dt + \int_A\sigma_n(t,X^n_t,\mu^n_t,\iota_n(a))N(da,dt), \ X^n_0 = X_0.
\]
Let $\Lambda^n$ denote the image of $\Lambda$ under the map $(t,a) \mapsto (t,\iota_n(a))$. Then $Q'(\Lambda^n \in \V[A_n])=1$, and it is easy to check that $P'_n := Q' \circ (\Lambda^n,X^n)^{-1}$ is in $\RC_n(\mu^n)$. Note that $\Lambda^n \rightarrow \Lambda$ holds $Q'$-a.s., and we will show also that $\E^{Q'}\|X^{n_k}-X\|^p_T \rightarrow 0$. To this end, note that
\begin{align*}
X^n_t - X_t = &\int_{[0,t] \times A}b_n(s,X^n_s,\mu^n_s,\iota_n(a)) - b_n(s,X_s,\mu^n_s,\iota_n(a))\Lambda_s(da)ds \\
	&+ \int_{[0,t] \times A}b_n(s,X_s,\mu^n_s,\iota_n(a)) - b(s,X_s,\mu_s,a)\Lambda_s(da)ds \\
	&+ \int_{[0,t] \times A}\sigma_n(s,X^n_s,\mu^n_s,\iota_n(a)) - \sigma_n(s,X_s,\mu^n_s,\iota_n(a))N(da,ds) \\
	&+ \int_{[0,t] \times A}\sigma_n(s,X_s,\mu^n_s,\iota_n(a)) - \sigma(s,X_s,\mu_s,a)N(da,ds).
\end{align*}
Use Jensen's inequality, the Lipschitz estimate, and the Burkholder-Davis-Gundy inequality to find a constant $C > 0$, independent of $n$, such that
\begin{align}
\E^{Q'}\|X^n - X\|_t^p \le &C\E^{Q'}\int_0^t\|X^n-X\|_s^pds \nonumber \\
	&+ C\E^{Q'}\int_{[0,t] \times A}\left|b_n(s,X_s,\mu^n_s,\iota_n(a)) - b(s,X_s,\mu_s,a)\right|^p\Lambda_s(da)ds \nonumber \\
	&+ C\E^{Q'}\left[\left(\int_0^t\|X^n-X\|_s^2ds\right)^{p/2}\right] \nonumber \\
	&+ C\E^{Q'}\left[\left(\int_{[0,t] \times A}\left|\sigma_n(s,X_s,\mu^n_s,\iota_n(a)) - \sigma(s,X_s,\mu_s,a)\right|^2\Lambda_s(da)ds\right)^{p/2}\right]. \nonumber
\end{align}
Let us label these terms $I_n$, $II_n$, $III_n$, and $IV_n$.
Recall that $p \ge 1$, by assumption (A.5).
If $p \ge 2$, note that $III_n \le CI_n$, for some new constant $C$.
On the other hand, if $p \in [1,2)$, we use Young's inequality in the form of $|xy| \le \epsilon^q|x|^q/q + \epsilon^{-q'}|y|^{q'}/q'$, where $q = 2/(2-p)$, $q'=2/p$, and $\epsilon > 0$. We deduce that
\begin{align*}
\E^{Q'}\left[\left(\int_0^t\|X^n-X\|_s^2ds\right)^{p/2}\right] &\le \E^{Q'}\left[\|X^n-X\|_t^{(2-p)p/2}\left(\int_0^t\|X^n-X\|_s^pds\right)^{p/2}\right] \\
	&\le \epsilon^{\frac{2}{2-p}}\left(1-\frac{p}{2}\right)\E^{Q'}\|X^n-X\|^p_t + \frac{p}{2\epsilon^{p/2}}\E^{Q'}\int_0^t\|X^n-X\|^p_sds.
\end{align*}
By choosing $\epsilon$ sufficiently small, we deduce
\[
\E^{Q'}\|X^n - X\|_t^p \le C(I_n + II_n + IV_n),
\]
for a new constant $C$. Now, once we show that $II_{n_k}$ and $IV_{n_k}$ tend to zero, we may conclude from Gronwall's inequality that $\E^{Q'}\|X^{n_k}-X\|^p_T \rightarrow 0$.
Since $|\iota_n(a)| \le |a|$ for all $a \in A$, there is another constant (again called) $C$ such that
\begin{align*}
\int_0^tds\int_A\Lambda_s(da)&\left|b_n(t,X_t,\mu^n_t,\iota_n(a)) - b(t,X_t,\mu_t,a)\right|^p \\
	&\le C\left(1 + \|X\|_T^p + \|\mu^n\|^p_T + \|\mu\|^p_T + \int_0^Tdt|\Lambda_t|^p\right),
\end{align*}
and similarly for the term involving $\sigma$, using $p_\sigma \le 2$ as in the proof of Lemma \ref{le:stateestimate}. Lemma \ref{le:stateestimate} implies that the right side above is $Q'$-integrable, and recall from \eqref{pf:unbounded4} that $\sup_n\|\mu^n\|^p_T  < \infty$. Since $\mu^{n_k} \rightarrow \mu$ and $\iota_n(a) \rightarrow a$ for each $a \in A$, the dominated convergence theorem shows that $II_{n_k}$ and $IV_{n_k}$ tend to zero.

With the convergence $\E^{Q'}\|X^{n_k}-X\|^p_T \rightarrow 0$ now established, the proof of the Lemma is nearly complete.
Note that $\|X^{n_k}\|_T^p$ are uniformly $Q'$-integrable, and
\[
\int_0^Tdt|\Lambda^n_t|^{p'} \le \int_0^Tdt|\Lambda_t|^{p'},
\]
and the latter is $Q'$-integrable, as in \eqref{pf:lambdaintegrable2}. Assumption (A.3) and \eqref{pf:unbounded4} then imply that both $g(X^n_T,\mu^n_T)$ and
\[
\int_0^Tdt\int_A\Lambda^n_t(da)f(t,X^n_t,\mu^n_t,a) = \int_0^Tdt\int_A\Lambda_t(da)f_n(t,X^n_t,\mu^n_t,\iota_n(a))
\]
are uniformly $Q'$-integrable. Since $\mu^{n_k} \rightarrow \mu$ and $\iota_n(a) \rightarrow a$, assumption (A.1) (continuity of $f$ and $g$) and the convergence of $X^{n_k}$ imply that
\[
g(X^{n_k}_T,\mu^{n_k}_T) - g(X_T,\mu_T) + \int_0^Tdt\int_A\Lambda_t(da)\left(f(t,X^{n_k}_t,\mu^{n_k}_t,\iota_{n_k}(a)) - f(t,X_t,\mu_t,a)\right) \rightarrow 0
\]
in $Q'$-measure. Now \eqref{pf:optimal3} follows from the dominated convergence theorem, after a transformation to the space $(\Omega',F'_t,Q')$:
\begin{align*}
J[f,g,A](\mu,P') &= \E^{Q'}\left[g(X_T,\mu_T) + \int_0^Tdt\int_A\Lambda_t(da)f(t,X_t,\mu_t,a) \right] \\
	&=\lim_{k\rightarrow\infty}\E^{Q'}\left[g(X^{n_k}_T,\mu^{n_k}_T) + \int_0^Tdt\int_A\Lambda_t(da)f(t,X^{n_k}_t,\mu^{n_k}_t,\iota_{n_k}(a)) \right] \\
	&= \lim_{k\rightarrow\infty}J[f,g,A_{n_k}](\mu^{n_k},P'_{n_k}).
\end{align*}
\end{proof}

\begin{proof}[Proof of Theorem \ref{th:relaxedexistence}]
Fix $P' \in \RC[b,\sigma,A]$. Find $P'_n$ as in Lemma \ref{le:limp'}. Optimality of $P_n$ for each $n$ imlies that
\[
J[f_n,g_n,A_n](\mu^n,P'_n) \le J[f_n,g_n,A_n](\mu^n,P_n).
\]
Use Lemma \ref{le:limp'} and the upper semicontinuity of $J$ (see Lemma \ref{le:Jcontinuous}) to get
\begin{align*}
J[f,g,A](\mu,P)	&\ge \limsup_{k\rightarrow\infty}J[f_{n_k},g_{n_k},A_{n_k}](\mu^{n_k},P_{n_k}) \\
	&\ge \lim_{k\rightarrow\infty}J[f_{n_k},g_{n_k},A_{n_k}](\mu^{n_k},P'_{n_k}) \\
	&= J[f,g,A](\mu,P').
\end{align*}
Since $P'$ was arbitrary, this implies that $P$ is optimal, or $P \in \RC^*[b,\sigma,f,g,A](\mu)$. Since also $P = P \circ X^{-1}$ by Lemma \ref{le:limitpoint}, it follows that $P$ is a relaxed MFG solution.
\end{proof}

\section{The elliptic case} \label{se:elliptic}
In this section, we see how to refine the results when the volatility is uncontrolled and uniformly nondegenerate. Notably, this allows us to relax the requirement that $b$ and $f$ are continuous in $x$ to mere measurability, and we may weaken somewhat the continuity requirement regarding the measure argument as well. We shall not overcomplicate the discussion by seeking the sharpest possible assumptions; instead, we build on the old but well known results of Stroock and Varadhan \cite{stroockvaradhanbook}. Define 
\begin{align*}
\P^p_L(\R^d) &:= \left\{\mu \in \P^p(\R^d) : \mu \ll \text{Lebesgue}\right\}, \\
\P^p_L(\C^d) &:= \left\{\mu \in \P^p(\C^d) : \mu_t \ll \text{Lebesgue}, \ \forall t \in (0,T]\right\}.
\end{align*}
The nondegeneracy of the volatility will ensure that the law of the solution of the control state equation will always lie in $\P^p_L(\C^d)$; note that we exclude $t=0$ in the definition of $\P^p_L(\C^d)$, to account for initial distributions which are not absolutely continuous.
We now assume the data are of the following form:
\begin{align*}
b &: [0,T] \times \R^d \times \P^p_L(\R^d) \times A \rightarrow \R^d, \\ 
\sigma &: [0,T] \times \R^d \times \P^p_L(\R^d) \rightarrow \R^{d \times m}, \\ 
f &: [0,T] \times \R^d \times \P^p_L(\R^d) \times A \rightarrow \R, \\ 
g &: \R^d \times \P^p_L(\R^d) \rightarrow \R.
\end{align*}
For each $r > 0$, let $B_r$ denote the centered closed ball of radius $r$ in $\R^d$. We work under the following assumptions:

\begin{assumption}{\textbf{(C)}} \label{assumption:C} {\ }
\begin{enumerate}
\item[(C.1)] The functions $b$, $\sigma$, $f$, and $g$ are jointly measurable. Moreover, $g=g(x,\mu)$ and $\sigma=\sigma(t,x,\mu)$ are continuous in $(x,\mu)$, uniformly in $t$. For each $r > 0$, the functions $b=b(t,x,\mu,a)$ and $f=f(t,x,\mu,a)$ and are continuous in $(\mu,a)$, uniformly in $(t,x) \in [0,T] \times B_r$, in the sense that
\[
\lim_{n\rightarrow\infty}\sup_{(t,x) \in [0,T] \times B_r}\left|(b,f)(t,x,\mu_n,a_n) - (b,f)(t,x,\mu,a)\right| = 0, \ \forall r > 0,
\]
whenever $(\mu_n,a_n) \rightarrow (\mu,a)$ in $\P^p_L(\R^d) \times A$.
\item[(C.2)] There exist $c_1 > 0$ such that, for all $(t,\mu,x) \in [0,T] \times \P^p_L(\R^d) \times \R^d$,
\begin{align*}
\sigma\sigma^\top(t,x,\mu) &\ge 1/c_1, \\
|b(t,x,\mu,a)| &\le c_1\left[1 + |x| + \left(\int_{\R^d}|z|^p\mu(dz)\right)^{1/p}\right], \\
|\sigma\sigma^\top(t,x,\mu)| &\le c_1\left[1 + |x|^2 + \left(\int_{\R^d}|z|^p\mu(dz)\right)^{2/p}\right]
\end{align*}
\item[(C.3)] There exists $c_2 > 0$ such that, for each $(t,x,\mu,a) \in [0,T] \times \R^d \times \P^p_L(\R^d) \times A$, 
\begin{align*}
|g(x,\mu)| &\le c_2\left(1 + |x|^p + |\mu|^p\right), \\
|f(t,x,\mu,a)| &\le c_2\left(1 + |x|^p + |\mu|^p\right).
\end{align*}
\item[(C.4)] The control space $A$ is a compact metric space.
\item[(C.5)] The initial distribution $\lambda$ is in $\P^{p'}(\R^d)$, and the exponents satisfy $p' > p \ge 2$.
\end{enumerate}
\end{assumption}

Notice that $b$ and $f$ need not be continuous in $x$, and the data only needs to be defined and continuous on $\P^p_L(\R^d)$, not all of $\P^p(\R^d)$. For example, this allows for rank-dependent data, such as $f(t,x,\mu,a) = \tilde{f}(t,\mu(-\infty,x],a)$. Rank-dependence poses a threat only to the assumption (C.1), but the uniform continuity can be checked easily using a well known theorem of P\'olya, which says that if $\mu_n \rightarrow \mu$ weakly with $\mu \in \P_L(\R)$, then $\mu_n(-\infty,x] \rightarrow \mu(-\infty,x]$ uniformly in $x$. For data depending on more general functionals of the form $(x,\mu) \mapsto \int\phi(x,y)\mu(dy)$, where $\phi$ is discontinuous, uniform weak convergence results as in \cite{rao-weakuniformconvergence} are useful for checking the uniform continuity assumption (C.1). 

Since $\sigma$ does not depend on the control, Proposition \ref{pr:sderepresentation} now takes a simpler form:

\begin{proposition} \label{pr:sderepresentation-simpler}
For $\mu \in \P^p_L(\C^d)$, $\RC(\mu)$ is precisely the set of laws $P' \circ (\Lambda,X)^{-1}$, where:
\begin{enumerate}
\item $(\Omega',\F'_t,P')$ is a filtered probability space supporting a $d$-dimensional $\F'_t$-adapted process $X$, an $m$-dimensional $\F'_t$-Wiener process $W$, and an $\F'_t$-predictable $\P(A)$-valued process $\Lambda$.
\item $P' \circ X_0^{-1} = \lambda$.
\item The state equation holds:
\begin{align}
dX_t &= \int_Ab(t,X_t,\mu_t,a)\Lambda_t(da)dt + \sigma(t,X_t,\mu_t)dW_t. \label{def:SDE-simpler}
\end{align}
\end{enumerate}
\end{proposition}

The goal of this section is to establish the following theorem:

\begin{theorem} \label{th:ellipticexistence}
Under assumption \ref{assumption:C}, there exists a relaxed Markovian MFG solution. If also \ref{assumption:convex} holds, then there exists a strict Markovian MFG solution.
\end{theorem}

Theorem \ref{th:markovselection} holds in this setting, as the proof did not use continuity (see Remark \ref{re:markovselection2}). Hence, we need to prove only that there exists a relaxed MFG solution under assumption \ref{assumption:C}. This could perhaps be done from the ground up, following the fixed point argument of \ref{se:boundedcoefficients}, but it seems simpler to take advantage of our previous existence theorem. The key ideas are to work \emph{only with Markovian controls} and to approximate the data $(b,\sigma,f,g)$ in an appropriate sense by a sequence of data $(b_n,\sigma_n,f_n,g_n)$, each of which satisfies assumption \ref{assumption:A}. Fix from now on a function $\psi \in C^\infty_0(\R^d)$ supported in the closed unit ball $B_1$ satisfying $\psi \ge 0$ and $\int\psi(x)dx=1$. Define $\psi_n(x) := n^d\psi(nx)$. Given $\mu \in \P(\R^d)$ and $\phi \in C^\infty_0(\R^d)$, define the convolution $\phi * \mu \in \P(\R^d)$ by
\[
(\phi * \mu)(dx) := \int_{\R^d}\phi(x-y)\mu(dy)dx.
\]
Define the data $(b_n,\sigma_n,f_n,g_n)$ as follows:
\begin{align*}
b_n(t,x,\mu,a) &:= \int_{\R^d}\psi_n(x-y)b(t,y,\psi_n * \mu,a)dy, \\
f_n(t,x,\mu,a) &:= \int_{\R^d}\psi_n(x-y)f(t,y,\psi_n * \mu,a)dy, \\
\sigma_n(t,x,\mu) &:= \int_{\R^d}\psi_n(x-y)\sigma(t,y,\psi_n * \mu)dy, \\
g(x,\mu) &:= g(x,\psi_n * \mu).
\end{align*}
Note that $(b_n,\sigma_n,f_n,g_n)$ are defined on all of $\P^p(\R^d)$, and not just $\P^p_L(\R^d)$. Moreover, $b_n$ and $\sigma_n$ are Lipschitz in $x$, uniformly in $(t,\mu,a)$, and in fact $(b_n,\sigma_n,f_n,g_n,A)$ satisfy assumptions \ref{assumption:A} and \ref{assumption:B}. For each $n$, by Theorem \ref{th:relaxedexistence} (or \ref{th:boundedexistence}), there exists a relaxed Markovian MFG solution corresponding to the data $(b_n,\sigma_n,f_n,g_n)$. That is, there exist $\mu^n \in \P^p(\C^d)$ and $P_n \in \RC^*[b_n,\sigma_n,f_n,g_n,A](\mu^n)$ such that $P_n \circ X^{-1} = \mu^n$. Moreover, there exists a measurable function $\hat{q}_n : [0,T] \times \R^d \rightarrow \P(A)$ such that
\begin{align}
P^n = \mu^n \circ \left(dt[\hat{q}_n(t,X_t)](da),X\right)^{-1}. \label{pf:elliptic1}
\end{align}

\subsection{Relative compactness of the approximations}
Analogously to Theorem \ref{th:relaxedexistence}, Theorem \ref{th:ellipticexistence} is proven by showing that $\mu^n$ are relatively compact and that in a sense each limit point gives rise to a MFG solution.

\begin{lemma} \label{le:elliptictight}
$\mu^n$ are relatively compact in $\P^p(\C^d)$, and each limit point is in $\P^p_L(\C^d)$.
\end{lemma}
\begin{proof}
By Lemma \ref{le:stateestimate}, we have (as in Lemma \ref{le:tight+moments})
\begin{align}
\sup_n\E^{P_n}\|X\|^{p'}_T = \sup_n\|\mu^n\|^{p'}_T < \infty. \label{pf:elliptictight1}
\end{align}
Hence, $P_n$ are relatively compact in $\P^p(\Omega)$, by Proposition \ref{pr:itocompact} and so $\mu^n$ are relatively compact in $\P^p(\C^d)$. Define
\begin{align*}
b^n(t,x) &:= \int_Ab_n(t,x,\mu^n_t,a)[\hat{q}_n(t,x)](da), \\
c^n(t,x) &:= \sigma_n\sigma_n^\top(t,x,\mu^n_t).
\end{align*}
Assumption (C.2) implies that $b^n$ and $c^n$ are locally uniformly bounded, in the sense that 
\[
\sup_n\sup_{(t,x) \in [0,T] \times B_r}|b^n(t,x)| + |c^n(t,x)| < \infty, \text{ for each } r > 0.
\]
Therefore the sequence $(b^n,c^n)$ admits a weak limit in $L^2_\text{loc}$; in particular, we may find a subsequence $n_k$ and functions $\widetilde{b} : [0,T] \times \R^d \rightarrow \R^d$ and $\widetilde{c} : [0,T] \times \R^d \rightarrow \R^{d \times d}$ such that
\begin{align*}
\lim_{k\rightarrow\infty}\int_0^T\int_{\R^d}b^{n_k}_i(t,x)\phi(t,x)dxdt &:= \int_0^T\int_{\R^d}\widetilde{b}_i(t,x)\phi(t,x)dxdt, \\
\lim_{k\rightarrow\infty}\int_0^T\int_{\R^d}c^{n_k}_{i,j}(t,x)\phi(t,x)dxdt &:= \int_0^T\int_{\R^d}\widetilde{c}_{i,j}(t,x)\phi(t,x)dxdt,
\end{align*}
for each $i,j=1,\ldots,d$ and $\phi \in C^\infty_0([0,T] \times \R^d)$. The functions $\widetilde{b}$ and $\widetilde{c}$ necessarily satisfy the same local bounds as $b^n$ and $c^n$, and also
\[
\widetilde{c}(t,x) \ge 1/c_1, \ \forall (t,x) \in [0,T] \times \R^d.
\]
Since $\{\mu^n_t : t \in [0,T], \ n \ge 1\}$ is relatively compact, the functions $\{c^n(t,\cdot) : t \in [0,T], \ n \ge 1\}$ are equicontinuous by assumption (C.1). We may thus assume $\widetilde{c}(t,\cdot)$ is continuous, uniformly in $t$, by the Arzel\`a-Ascoli theorem. It follows that the martingale problem corresponding to $(\widetilde{b},\widetilde{c})$ is well-posed, and from \cite[Theorem 11.3.4]{stroockvaradhanbook} we conclude that $\mu^{n_k}$ converges to the unique probability measure $\mu$ on $\C^d$ such that $\mu \circ X_0^{-1} = \lambda$ and such that
\[
\phi(X_t) - \int_0^t\left(\widetilde{b}(s,X_s)^\top D\phi(X_s) + \frac{1}{2}\text{Tr}\left[\widetilde{c}(s,X_s)D^2\phi(X_s)\right]\right)ds
\]
is a $\mu$-martingale for each $\phi \in C^\infty_0(\R^d)$. It is shown in \cite[Corollary 9.1.10]{stroockvaradhanbook} that $\mu_t$ admits a density for each $t > 0$.
\end{proof}

\subsection{Relaxed Markovian controls}
From now on, fix a limit point $\mu \in \P^p(\C^d)$ of $\mu^n$, and we will abuse notation by assuming $\mu^n \rightarrow \mu$ itself. Consider the set $\M^0$ of \emph{Markovian controls}, defined to be the set of all measurable functions from $[0,T] \times \R^d$ to $\P(A)$. Let $L_n = L[b_n,\sigma_n,A]$ denote the generator associate to $(b_n,\sigma_n,A)$. Given $q \in \M^0$ and $\phi \in C^\infty_0(\R^d)$, define $M_t[q,\phi] : \C^d \rightarrow \R$ and $M^n_t[q,\phi] : \C^d \rightarrow \R$ by
\begin{align*}
M_t[q,\phi](x) := \phi(x_t) - \int_0^tds\int_A[q(s,x_s)](da)L\phi(s,x_s,\mu_s,a), \\
M^n_t[q,\phi](x) := \phi(x_t) - \int_0^tds\int_A[q(s,x_s)](da)L_n\phi(s,x_s,\mu^n_s,a).
\end{align*}
Let $\F^X_t$ denote the natural filtration on $\C^d$. The classical results of Stroock and Varadhan \cite{stroockvaradhanbook} ensure that each of these martingale problems are well-posed. That is, for each $q \in \M^0$, there is a unique $Q[q] \in \P^p(\C^d)$ such that $Q[q] \circ X_0^{-1} = \lambda$ and such that $M_t[q,\phi]$ is a $Q[q]$-martingale for each $\phi \in C^\infty_0(\R^d)$. Similarly, there is a unique $Q_n[q] \in \P^p(\C^d)$ such that $Q_n[q] \circ X_0^{-1} = \lambda$ and such that $M^n_t[q,\phi]$ is a $Q_n[q]$-martingale for each $\phi \in C^\infty_0(\R^d)$.
Define also
\begin{align*}
\widetilde{Q}[q] &:= Q[q] \circ \left(dt[q(t,X_t)](da),X\right)^{-1}, \\
\widetilde{Q}_n[q] &:= Q_n[q] \circ \left(dt[q(t,X_t)](da),X\right)^{-1}.
\end{align*}
Note that $\widetilde{Q}[q] \in \RC[b,\sigma,A](\mu)$ and $\widetilde{Q}_n[q] \in \RC[b_n,\sigma_n,A](\mu^n)$ for each $q \in \M^0$. In fact, it follows from \cite[Corollary 9.1.0]{stroockvaradhanbook} that the measures $Q[q]_t$ and $Q_n[q]_t$ on $\R^d$ admit densities (with respect to Lebesgue measure) for each $q \in \M^0$ and each $t > 0$. Thus, if $q=q'$ for Lebesgue-almost-every $(t,x) \in [0,T] \times \R^d$, we have $Q[q]=Q[q']$ and $Q_n[q]=Q_n[q']$. Thus, if $\M$ is defined to be the quotient space of equivalence classes of a.e. equal elements of $\M^0$, we may define $Q[q]$, $Q_n[q]$, $\widetilde{Q}[q]$, and $\widetilde{Q}_n[q]$ unambiguously for each $q \in \M$.

Fix arbitrarily some strictly positive probability density $\Phi$ on $\R^d$, e.g. a Gaussian. We may identify $q \in \M$ with the measure
\[
dtdx\Phi(x)[q(t,x)](da)
\]
on $[0,T] \times \R^d \times A$, and conversely for any measure on $[0,T] \times \R^d \times A$ with $[0,T] \times \R^d$-marginal equal to $dtdx\Phi(x)$ there exists a unique corresponding $q \in \M$, by disintegration. Thus we may topologize $\M$ by transferring the weak convergence topology from the space of measures on $[0,T] \times \R^d \times A$. This means $q_n \rightarrow q$ if and only if
\[
\lim_{n\rightarrow\infty}\int_0^Tdt\int_{\R^d}dx\Phi(x)\int_A[q_n(t,x)-q(t,x)](da)\phi(t,x,a) = 0,
\]
for all bounded continuous $\phi$. Since $A$ is a compact metric space, so is $\M$.

\subsection{Passage to the limit}
After two lemmas, we will prove the crucial Proposition \ref{pr:Qcontinuous}, explaining the convergence of $\widetilde{Q}_n[q_n]$. With this proposition in hand, the proof of Theorem \ref{th:ellipticexistence} will be straightforward.

\begin{lemma} \label{le:convolutionuniform}
If $\nu_n \rightarrow \nu$ in $\P^p(\R^d)$, then $\psi_n * \nu_n \rightarrow \nu$ in $\P^p(\R^d)$. If $\nu^n \rightarrow \nu$ in $\P^p(\C^d)$, then $\psi_n * \nu^n_t \rightarrow \nu_t$ in $\P^p(\R^d)$, uniformly in $t \in [0,T]$.
\end{lemma}
\begin{proof}
For $\nu,\eta \in \P^p(\C^d)$, it is clear that
\[
\sup_{t \in [0,T]}d^p_{\R^d}(\nu_t,\eta_t) \le d^p_{\C^d}(\nu,\eta).
\]
In particular, the function $\P^p(\C^d) \ni \nu \mapsto (\nu_t)_{t \in [0,T]} \in C([0,T];\P^p(\R^d))$ is uniformly continuous. 
Define the sequence of functions $F_n : \P^p(\R^d) \rightarrow \P^p(\R^d)$ by
\[
F_n(\nu) := \psi_n * \nu.
\]
It is well-known that $F_n$ converges pointwise to the identity. Actually, both claims follow from the simple fact that $F_n$ converges \emph{uniformly}. Indeed, to estimate $d^p_{\R^d}(F_n(\nu),\nu)$, define the following coupling of the laws $\nu$ and $\psi_n * \nu$: Construct on some probability space two independent random vectors $Y$ and $Z$ with respective laws $\nu$ and $\psi_n(x)dx$. Then $Y+Z$ has law $\psi_n * \nu$, and so
\[
d^p_{\R^d}(F_n(\nu),\nu) \le \E\left[|(Z+Y) - Y|^p\right] = \int_{\R^d}|x|^p\psi_n(x)dx.
\]
\end{proof}

\begin{lemma} \label{le:ellipticL1}
If $\nu^n \in \P^p(\C^d)$, $\nu \in \P^p_L(\C^d)$, and $\nu^n \rightarrow \nu$ in $\P^p(\C^d)$, then for each $r > 0$
\[
\lim_{n\rightarrow\infty}\int_0^Tdt\int_{B_r}dx\sup_{a \in A}\left|b_n(t,x,\nu^n_t,a) - b(t,x,\nu_t,a)\right| = 0,
\]
and an analogous results hold with $f_n$ or $\sigma_n$ in place of $b_n$.
\end{lemma}
\begin{proof}
Since $t \mapsto \nu_t$ is continuous, $\{\nu_t : t \in [0,T]\}$ is compact in $\P^p(\R^d)$. 
Lemma \ref{le:convolutionuniform} implies that the set $\{\psi_n * \nu^n_t : t \in [0,T], \ n \ge 1\}$ is relatively compact in $\P^p(\R^d)$. By assumption (C.2),
\[
\sup_n\sup_{(t,x,a) \in (0,T] \times B_{r} \times A}\left(|b(t,x,\psi_n * \nu^n_t,a)| + |b(t,x,\nu_t,a)|\right) < \infty.
\]
Compactness of $A$ and assumption (C.1) imply that $b(t,x,\nu,a)$ is continuous in $\nu$, uniformly in $(t,x,a) \in [0,T] \times B_{r+1} \times A$. Along with Lemma \ref{le:convolutionuniform}, this implies
\[
C_n := \sup_{(t,y,a) \in (0,T] \times B_{r+1} \times A}\left|b(t,y,\psi_n * \nu^n_t,a) - b(t,y,\nu_t,a)\right| \rightarrow 0.
\]
Thus
\begin{align*}
\sup_{a \in A}&\left|b_n(t,x,\nu^n_t,a) - b(t,x,\nu_t,a)\right| \\
	&= \sup_{a \in A}\left|\int_{\R^d}\psi_n(x-y)b(t,y,\psi_n * \nu^n_t,a)dy - b(t,x,\nu_t,a)\right| \\
	&\le C_n + \sup_{a \in A}\left|\int_{\R^d}\psi_n(x-y)b(t,y,\nu_t,a)dy - b(t,x,\nu_t,a)\right|.
\end{align*}
It remains to show that
\begin{align}
\lim_{n\rightarrow\infty}\int_0^Tdt\int_{B_r}dx\sup_{a \in A}\left|\int_{\R^d}\psi_n(x-y)b(t,y,\nu_t,a)dy - b(t,x,\nu_t,a)\right| = 0, \label{pf:ellipticL1-1}
\end{align}
The set $K := \{\nu_t : t \in (0,T]\} \subset \P^p_L(\R^d)$ is relatively compact, and $b(t,x,\nu,a)$ is continuous in $(\nu,a)$ uniformly in $(t,x) \in [0,T] \times B_r$ by (C.1). Thus, for any $\epsilon > 0$ there exists a finite collection of elements $(\eta_m,a_m)_m$ of $K \times A$ such that for each $(\eta,a) \in K \times A$ there exists $m$ such that
\[
\sup_{(t,x) \in (0,T] \times B_{r+1}}\left|b(t,x,\eta,a) - b(t,x,\eta_m,a_m)\right| < \epsilon.
\]
Then, if $|B_r|$ denotes the Lebesgue measure of $B_r$,
\begin{align*}
\int_0^Tdt&\int_{B_r}dx\sup_{a \in A}\left|\int_{\R^d}\psi_n(x-y)b(t,y,\nu_t,a)dy - b(t,x,\nu_t,a)\right| \\
	&\le 2\epsilon T|B_r| + \sum_m\int_0^Tdt\int_{B_r}dx\left|\int_{\R^d}\psi_n(x-y)b(t,y,\eta_m,a_m)dy - b(t,x,\eta_m,a_m)\right|.
\end{align*}
Since the summation is finite, sending $n \rightarrow\infty$ and then $\epsilon \downarrow 0$ proves \eqref{pf:ellipticL1-1}.
\end{proof}

\begin{proposition} \label{pr:Qcontinuous}
If $q_n \rightarrow q $ in $\M$, then $Q_n[q_n] \rightarrow Q[q]$ in $\P^p(\C^d)$ and
\begin{align}
\lim_{n\rightarrow\infty}J[f_n,g_n](\mu^n,\widetilde{Q}_n[q_n]) = J[f,g](\mu,\widetilde{Q}[q]). \label{pf:Qcontinuous0}
\end{align}
\end{proposition}
\begin{proof}
Lemma \ref{le:ellipticL1} implies
\[
\lim_{n\rightarrow\infty}\int_0^Tdt\int_{\R^d}dx\phi(t,x)\left[\sigma_n\sigma_n^\top(t,x,\mu^n_t)-\sigma\sigma^\top(t,x,\mu_t)\right] = 0,
\]
for each $\phi \in C^\infty_0([0,\infty) \times \R^d)$. It will follow from \cite[Theorem 11.3.4]{stroockvaradhanbook} that $Q_n[q_n] \rightarrow Q[q]$ if we show that
\begin{align}
\lim_{n\rightarrow\infty}\int_0^Tdt\int_{\R^d}dx\phi(t,x)\left[\int_A[q_n(t,x)](da)b_n(t,x,\mu^n_t,a) - \int_A[q(t,x)](da)b(t,x,\mu_t,a)\right] = 0, \label{pf:Qcontinuous1}
\end{align}
for each $\phi \in C^\infty_0([0,\infty) \times \R^d)$. First, use Lemma \ref{le:ellipticL1} to conclude that
\[
\lim_{n\rightarrow\infty}\int_0^Tdt\int_{\R^d}dx\phi(t,x)\int_A[q_n(t,x)](da)\left[b_n(t,x,\mu^n_t,a) - b(t,x,\mu_t,a)\right] = 0.
\]
On the other hand, Lemma \ref{le:componentwise} (with $p=0$) and the bounded convergence theorem yield
\[
\lim_{n\rightarrow\infty}\int_0^Tdt\int_{\R^d}dx\phi(t,x)\int_A[q_n(t,x) - q(t,x)](da)b(t,x,\mu_t,a) = 0.
\]
Now to prove \eqref{pf:Qcontinuous0}, we must show that
\begin{align}
\lim_{n\rightarrow\infty}\E^{Q_n[q_n]}&\left[\int_0^Tdt\int_A[q_n(t,X_t)](da)f_n(t,X_t,\mu^n_t,a) + g(X_T,\mu^n_T)\right] \nonumber  \\
	= \E^{Q[q]}&\left[\int_0^Tdt\int_A[q(t,X_t)](da)f(t,X_t,\mu_t,a)+ g(X_T,\mu_T)\right]. \label{pf:Qcontinuous2}
\end{align}
First, since $\mu^n_T \rightarrow \mu_T$ in $\P^p(\R^d)$ and $Q_n[q_n]\rightarrow Q[q]$ in $\P^p(\C^d)$, joint continuity of $g$ implies
\begin{align*}
\lim_{n\rightarrow\infty}\E^{Q_n[q_n]}\left[g(X_T,\mu^n_T)\right] = \E^{Q[q]}\left[g(X_T,\mu_T)\right].
\end{align*}
On the other hand, it is proven in the same manner as \eqref{pf:Qcontinuous1} that
\[
\lim_{n\rightarrow\infty}\int_0^Tdt\int_{\R^d}dx\phi(t,x)\left[\int_A[q_n(t,x)](da)f_n(t,x,\mu^n_t,a) - \int_A[q(t,x)](da)f(t,x,\mu_t,a)\right] = 0,
\]
for each $\phi \in C^\infty_0([0,\infty) \times \R^d)$. Noting that the functions
\[
(t,x) \mapsto \int_A[q_n(t,X_t)](da)f_n(t,x,\mu^n_t,a)
\]
are uniformly bounded on $\{(t,x) : |x| \le r\}$, it follows from \cite[Lemma 11.3.2]{stroockvaradhanbook} that
\begin{align*}
\lim_{n\rightarrow\infty}\E^{Q_n[q_n]}&\left[1_{\{\|X\| \le r\}}\int_0^Tdt\int_A[q_n(t,X_t)](da)f_n(t,X_t,\mu^n_t,a)\right] \\
	= \E^{Q[q]}&\left[1_{\{\|X\| \le r\}}\int_0^Tdt\int_A[q(t,X_t)](da)f(t,X_t,\mu_t,a)\right].
\end{align*}
But since $\sup_n\E^{Q_n[q_n]}\|X\|^{p'}_T < \infty$ and $f$ has $p$-order growth in $X$, we also have
\[
\lim_{r \rightarrow \infty} \sup_n\E^{Q_n[q_n]}\left[1_{\{\|X\| > r\}}\int_0^Tdt\int_A[q_n(t,X_t)](da)f_n(t,X_t,\mu^n_t,a)\right] = 0.
\]
This completes the proof of \eqref{pf:Qcontinuous2}.
\end{proof}

\begin{proof}[Proof of Theorem \ref{th:ellipticexistence}]
Recall that $P_n \in \RC^*[b_n,\sigma_n,f_n,g_n,A](\mu^n)$ satisfies $P_n \circ X^{-1} = \mu^n$, and \eqref{pf:elliptic1}. In the newer notation, this means $\widetilde{Q}[\hat{q}_n] = P_n$ and $Q_n[\hat{q}_n] = \mu^n$ for each $n$. (We will supress the $A$ that should appear in the brackets following such notation as $\RC[b,\sigma] := \RC[b,\sigma,A]$, since $A$ will not vary.) 

Since $\M$ is compact, the sequence $\hat{q}_n$ has a limit point $\hat{q} \in \M$. As with $\mu^n$, we will abuse notation somewhat by assuming $\hat{q}_n \rightarrow \hat{q}$, so that now $(\mu^n,\hat{q}_n) \rightarrow (\mu,q)$, while really this is only along a subsequence. Define $P := \widetilde{Q}[\hat{q}]$, so that clearly $P \in \RC[b,\sigma](\mu)$. We will show that in fact $P$ is a relaxed MFG solution. Proposition \ref{pr:Qcontinuous} implies
\[
P \circ X^{-1} = Q[\hat{q}] = \lim_{n\rightarrow\infty}Q_n[\hat{q}_n] = \lim_{n\rightarrow\infty}\mu^n = \mu.
\]
It remains only to show that $P$ is optimal. Fix any $P' \in \RC[b,\sigma](\mu)$. By Theorem \ref{th:markovselection}, there exists a Markovian $P'_0 \in \RC[b,\sigma](\mu)$ with $J(\mu,P'_0) \ge J(\mu,P')$. That $P'_0$ is Markovian means that there exists $q_0 \in \M$ such that $P'_0 = \widetilde{Q}[q_0]$. Now set $P'_n := \widetilde{Q}_n[q_0]$, so that $P'_n \in \RC[b_n,\sigma_n](\mu^n)$. By Proposition \ref{pr:Qcontinuous}, we have both
\begin{align*}
\lim_{n\rightarrow\infty}J[f_n,g_n](\mu^n,P_n) &= J[f,g](\mu,P), \\
\lim_{n\rightarrow\infty}J[f_n,g_n](\mu^n,P'_n) &= J[f,g](\mu,P'_0).
\end{align*}
But for each $n$, $P_n$ is optimal, and so
\[
J[f_n,g_n](\mu^n,P_n) \ge J[f_n,g_n](\mu^n,P'_n).
\]
Therefore
\[
J[f,g](\mu,P) \ge J[f,g](\mu,P'_0) \ge J[f,g](\mu,P').
\]
This holds for all $P' \in \RC[b,\sigma](\mu)$, and thus $P \in \RC^*[b,\sigma,f,g](\mu)$.
\end{proof}

\section{A counterexample} \label{se:counterexample}
It is a bit disappointing that assumptions \ref{assumption:A} and \ref{assumption:C} both exclude linear-quadratic models with objectives which are quadratic in both $a$ and $x$. That is, we do not allow
\[
f(t,x,\mu,a) = -|a|^2 - c\left|x + c'\bar{\mu}\right|^2, \ c,c' \in \R,
\]
where we have abbreviated $\bar{\mu} := \int_\R z\mu(dz)$ for $\mu \in \P^1(\R)$. On the one hand, if $c < 0$ and $|c|$ is large enough, then it may hold for each $\mu$ that $\RC^*(\mu) = \emptyset$, and obviously non-existence of optimal controls prohibits the existence of MFG solutions. The goal now is to demonstrate that even when $f$ and $g$ are bounded from above, we cannot expect a general existence result if $p'=p$. We are certainly not the first to notice what can go wrong in linear-quadratic mean field games when the constants do not align properly; see, for example, \cite[Theorem 3.1]{carmonadelaruelachapelle-mkvvsmfg}. Of course, the refined analyses of \cite{bensoussan-lqmfg,carmonadelaruelachapelle-mkvvsmfg} give many positive results on linear-quadratic mean field games, but we simply wish to provide a tractable example of nonexistence to show that this edge case $p' = p$ requires more careful analysis.

Consider constant volatility $\sigma$, $d=1$, $p'=p=2$, $A = \R$, and and the following data:
\begin{align*}
b(t,x,\mu,a) &= a, \\
f(t,x,\mu,a) &= -a^2, \\
g(x,\mu) &= -(x + c\bar{\mu})^2, \ c \in \R.
\end{align*}
With great foresight, choose $T > 0$, $c \in \R$, and $\lambda \in \P^2(\R)$ such that
\begin{align*}
c=-(1+T)/T, \quad \text{ and } \quad \bar{\lambda} \neq 0.
\end{align*}
Assumption \ref{assumption:A} and \ref{assumption:convex} hold with the one exception that the assumption $p' > p$ is violated. Theorem \ref{th:markovselection} still applies (see Remark \ref{re:markovselection2}), and we conclude that if there exists a relaxed MFG solution, then there must exist a strict MFG solution. Suppose there exists a strict MFG solution $\mu$, so that we may find $P \in \RC^*(\mu)$ satisfying $P \circ X^{-1} = \mu$ and $P(\Lambda = dt\delta_{\alpha^*_t}(da)) = 1$ for some $\F_t$-progressive real-valued process $\alpha^*_t$ satisfying
\[
\E\int_0^1|\alpha^*_t|^2dt < \infty,
\]
where $\E$ denotes expectation under $P$. Denote by $W$ the $P$-Wiener process $W$ on $\Omega$ satisfying
\begin{align}
X_t = X_0 + \int_0^t\alpha^*_sds + \sigma W_t, \ t \in [0,T]. \label{def:counterexample11}
\end{align}
In particular, $\alpha^*$ is the unique minimizer among $\F_t$-progressive square-integrable real-valued processes $\alpha$ of
\begin{align*}
J(\alpha) := \E\left[\int_0^T|\alpha_t|^2dt + \left(X^\alpha_T + c\bar{\mu}_T\right)^2\right],
\end{align*}
where
\[
X^\alpha_t = X_0 + \int_0^t\alpha_sds + \sigma W_t, \ t \in [0,T].
\]
Expand the square
\[
\left(X^\alpha_T + c\bar{\mu}_T\right)^2 = \left(X_0 + \int_0^T\alpha_tdt + \sigma W_T + c\bar{\mu}_T\right)^2
\]
and discard the terms which do not involve $\alpha$ to see that minimizing $J(\alpha)$ is equivalent to minimizing
\[
\widetilde{J}(\alpha) = \E\left[\int_0^T\left[|\alpha_t|^2 + 2\left(X_0 + \sigma W_T + c\bar{\mu}_T\right)\alpha_t\right]dt + \left(\int_0^T\alpha_tdt\right)^2\right]
\]
Since $\alpha^*$ is the unique minimizer, for any other $\alpha$ it holds that
\begin{align*}
0 &= \frac{d}{d\epsilon}\left.\widetilde{J}\left(\alpha^* + \epsilon\alpha\right)\right|_{\epsilon=0} = 2\E\left[\int_0^T\left[\alpha_t\alpha^*_t + \left(X_0 + \sigma W_T + c\bar{\mu}_T\right)\alpha_t\right]dt + \int_0^T\alpha_tdt\int_0^T\alpha^*_tdt\right].
\end{align*}
In particular, if $\alpha$ is deterministic, then
\begin{align*}
0 &= \int_0^T\alpha_t\E\left[\alpha^*_t + X_0 + \sigma W_T + c\bar{\mu}_T + \int_0^T\alpha^*_sds\right]dt
\end{align*}
Since this holds for every deterministic square-integrable $\alpha$, it follows that
\begin{align*}
0 &= \E\left[\alpha^*_t + X_0 + \sigma W_T + c\bar{\mu}_T + \int_0^T\alpha^*_sds\right].
\end{align*}
Noting that $\bar{\mu}_0 = \E X_0$, we get
\[
-\E\alpha^*_t = \bar{\mu}_0 + c\bar{\mu}_T + \int_0^T\E\alpha^*_sds.
\]
In particular, $\E\alpha^*_t$ is constant in $t$. Defining $\bar{\alpha} = \E\alpha^*_t$ for all $t$, we must have
\[
\bar{\alpha} = -\frac{\bar{\mu}_0 + c\bar{\mu}_T}{1+T}.
\]
Take expectations in \eqref{def:counterexample11} to get $\bar{\mu}_t = \bar{\mu}_0 + \bar{\alpha}t$. But then
\begin{align*}
\bar{\mu}_T &= \bar{\mu}_0 + \bar{\alpha}T = \bar{\mu}_0 - \frac{\bar{\mu}_0 + c\bar{\mu}_T}{1+T}T \\
	&= \frac{\bar{\mu}_0}{1+T} + \bar{\mu}_T,
\end{align*}
where in the last line we finally used the particular choice of $c=-(1+T)/T$. This implies $\bar{\mu}_0 = 0$, which contradicts $\bar{\lambda} \neq 0$ since $\bar{\mu}_0 = \bar{\lambda}$. Hence, for this particular choice of data, there is no solution.

\section{Future work} \label{se:extensions}
The ideas developed in this paper seems quite widely applicable to problems of mean field games, and even to mean-field type control problems (i.e. controlled McKean-Vlasov dynamics, as in \cite{bensoussan-mfgbook,carmonadelaruelachapelle-mkvvsmfg}). For example, concurrently with the finalization of this manuscript for publication, recent papers exploited relaxed control theory to study mean field games with common noise \cite{carmonadelaruelacker-mfgcommonnoise} as well as the convergence of equilbria of finite-player games \cite{fischer-mfgconnection,lacker-meanfieldlimit}. Presumably, the same method of studying mean field games via controlled martingale problems should be applicable to much more general types of state processes, such as jump-diffusions or even processes in abstract spaces. Indeed, the corresponding relaxed control theory is developed by El Karoui et al. in \cite[Section 8]{elkaroui-compactification} and by Kurtz and Stockbridge in \cite{kurtzstockbridge-1998}. The latter paper is impressively general in that the state process takes values in an arbitary complete, separable, and locally compact metric space, and various objective structures (finite time horizon, infinite time horizon, ergodic, and first passage time) are permitted. The latter generalization should certainly be adaptable to mean field games.

The difficulty with more general state processes will be ensuring some kind of uniqueness of the martingale problems, to permit lower hemicontinuity of the set-valued map $\mu \mapsto \RC[b,\sigma,A](\mu)$ (see Lemma \ref{le:rccontinuous}). Indeed, this was crucial in establishing upper hemicontinuity of the fixed point map in Lemma \ref{le:Jcontinuous} and was implicit in the approximation procedure of Lemma \ref{le:limp'}. In our case, uniqueness was found by exploiting the well-known relationship between the martingale problem and solutions of SDEs. With this in mind, it should not be difficult to handle jump-diffusions, for which the link between the martingale problem and the stochastic differential equation is fairly well-understood. But the SDEs (or martingale problems) involved in control problems inherently involve \emph{random coefficients}, for which our knowledge of uniqueness is mostly limited to Lipschitz coefficients; for this reason, it would presumably be difficult to proceed on the same level of abstraction as in \cite{kurtzstockbridge-1998}.

In the uniformly elliptic setting of Section \ref{se:elliptic}, we avoided the difficulties caused by random coefficients by working mostly with Markovian controls. The Markovian selection argument of Theorem \ref{th:markovselection} holds rather generally (see Theorem 5.1 and its corollaries in \cite{kurtzstockbridge-1998}), so there is some hope to extend the approach to more general situations.

\bibliographystyle{amsplain}
\bibliography{relaxedMFGbib}

\appendix

\section{Continuity in Wasserstein spaces} \label{ap:wasserstein}
Recall the definition of the Wasserstein metric from \eqref{def:wasserstein}. The definition $\P^0(E) := \P(E)$ will be a useful convention, and recall that $\P^0(E)$ is given the topology of weak convergence. Fix $p=0$ or $p \ge 1$ throughout the section. The following description of Wasserstein space is well-known and used implicitly throughout the paper.

\begin{proposition}[Theorem 7.12 of \cite{villanibook}] \label{pr:wasserstein}
Let $(E,\rho)$ be a metric space, and suppose $\mu,\mu_n \in \P^p(E)$. Then the following are equivalent
\begin{enumerate}
\item $d_{E,p}(\mu_n,\mu) \rightarrow 0$.
\item $\mu_n \rightarrow \mu$ weakly and for some (and thus any) $x_0 \in E$ we have
\begin{align}
\lim_{r \rightarrow \infty}\sup_n\int_{\{x : \rho^p(x,x_0) \ge r\}}\mu_n(dx)\rho^p(x,x_0) = 0. \label{def:uiwasserstein}
\end{align}
\item $\int\phi\,d\mu_n \rightarrow \int\phi\,d\mu$ for all continuous functions $\phi : E \rightarrow \R$ such that there exists $x_0 \in E$ and $c > 0$ for which $|\phi(x)| \le c(1 + \rho^p(x,x_0))$ for all $x \in E$.
\end{enumerate}
In particular, (2) implies that a sequence $(\mu_n)_n \subset \P^p(E)$ is relatively compact if and only it is tight (i.e. relatively compact in $\P(E)$) and satisfies \eqref{def:uiwasserstein}.
\end{proposition}

Properties of the weak convergence topology of $\P(E)$ are naturally transferred to $\P^p(E)$ via homeomorphism, and this is behind most of the results to follow: Fix  $x_0 \in E$, and define $\psi(x) := 1 + \rho^p(x,x_0)$. For each $\mu \in \P^p(E)$ define a measure $\psi\,\mu \in \P(E)$ by $\psi\,\mu(B) = \int_B\psi\,d\mu$ for all $B \in \B(E)$. Then $\mu \mapsto \psi\,\mu/\int\psi\,d\mu$ is easily seen to define a homeomorphism from $(\P^p(E),d_{E,p})$ to $\P(E)$ with the weak topology.

In the following two lemmas, let $(E,\rho_E)$ and $(F,\rho_F)$ be two complete separable metric spaces. Equip $E \times F$ with the metric formed by adding the metrics of $E$ and $F$, given by $((x_1,x_2),(y_1,y_2)) \mapsto \rho_1(x_1,y_1) + \rho_2(x_2,y_2)$, but this choice is inconsequential. The following few lemmas are all fairly well-known when $p=0$, and the general case is then proven using the homeomorphism.

\begin{lemma} \label{le:productrelcompactness}
A set $K \subset \P^p(E \times F)$ is relatively compact if and only if $\{P(\cdot \times F) : P \in K\} \subset \P^p(E)$ and $\{P(E \times \cdot) : P \in K\} \subset \P^p(F)$ are relatively compact.
\end{lemma}

\begin{lemma} \label{le:componentwise}
Let $\phi : E \times F \rightarrow \R$ satisfy the following:
\begin{enumerate}
\item $\phi(\cdot,y)$ is measurable for each $y \in F$.
\item $\phi(x,\cdot)$ is continuous for each $x \in E$.
\item There exist $c > 0$, $x_0 \in E$, and $y_0 \in E_2$ such that
\[
|\phi(x,y)| \le c(1 + \rho^p_E(x,x_0) + \rho^p_F(y,y_0)), \ \forall(x,y) \in E \times F.
\]
\end{enumerate}
If $P_n \rightarrow P$ in $\P^p(E \times F)$ and $P_n(\cdot \times F) = P(\cdot \times F)$ for all $n$, then $\int\phi\,dP_n \rightarrow \int\phi\,dP$.
\end{lemma}
\begin{proof}
The $p=0$ case was shown by Jacod and M\'emin in \cite{jacodmemin-stable}, and this extends to the general case using the homeomorphism described above.
\end{proof}

\begin{corollary} \label{co:usc}
Suppose $\phi$ satisfies (1) and (2) of Lemma \ref{le:componentwise}, and instead
\begin{enumerate}
\item[(3')] There exist $c > 0$, $x_0 \in E$, and $y_0 \in E_2$ such that
\[
\phi(x,y) \le c(1 + \rho^p_E(x,x_0) + \rho^p_F(y,y_0)), \ \forall(x,y) \in E \times F.
\]
\end{enumerate}
If $P_n \rightarrow P$ in $\P^p(E \times F)$ and $P_n(\cdot \times F) = P(\cdot \times F)$ for all $n$, then
\[
\limsup_{n\rightarrow\infty}\int\phi\,dP_n \le \int\phi\,dP.
\]
\end{corollary}
\begin{proof}
For each $M < 0$, Lemma \ref{le:componentwise} implies
\begin{align*}
\int M \vee \phi\,dP &= \lim_{n\rightarrow\infty}\int M \vee \phi\,dP_n \ge \limsup_{n\rightarrow\infty}\int \phi\,dP_n.
\end{align*}
Send $M \downarrow -\infty$ and use the monotone convergence theorem.
\end{proof}

The rest of the section specializes to the space $\V[A]$. Recall that $A$ is a closed subset of Euclidean space, $p \ge 1$, and $\V[A]$ is defined at the beginning of Section \ref{se:relaxed}. Note that Proposition \ref{pr:wasserstein} implies that $d_{\V[A]}$ metrizes weak convergence when $A$ is compact. 
The following Corollary \ref{co:V} was used in the proof of Lemma \ref{le:Jcontinuous}.

\begin{corollary} \label{co:V}
Let $(E,\rho)$ be a complete separable metric space. Let $\phi : [0,T] \times E \times A \rightarrow \R$ be measurable with $\phi(t,\cdot)$ jointly continuous for each $t \in [0,T]$. Suppose there exist $c > 0$ and $x_0 \in E$ such that one of the following holds:
\begin{enumerate}
\item $\phi(t,x,a) \le c(1 + \rho^p(x,x_0) + |a|^p)$, for all $(t,x,a)$.
\item $|\phi(t,x,a)| \le c(1 + \rho^p(x,x_0) + |a|^p)$, for all $(t,x,a)$.
\end{enumerate}
If (1) holds, then the following function is upper semicontinuous:
\[
C([0,T];E) \times \V[A] \ni (x,q) \mapsto \int q(dt,da)\phi(t,x(t),a).
\]
If (2) holds, then this function is continuous.
\end{corollary}
\begin{proof}
These follow immediate from Lemma \ref{le:componentwise} and Corollary \ref{co:usc}, after observing that the following function is jointly continuous:
\[
C([0,T];E) \times \V[A] \ni (x,q) \mapsto \frac{1}{T}q(dt,da)\delta_{x(t)}(de) \in \P^p([0,T] \times A \times E)
\]
\end{proof}

\section{Compactness in $\P^p(\P^p(E))$} \label{se:metawasserstein}
For $P \in \P(\P(E))$, define the mean measure $mP \in \P(E)$ by
\[
mP(C) := \int_{\P(E)}P(d\mu)\mu(C).
\]
The following proposition for $p=0$ may be found in Proposition 2.2(ii) of Sznitman \cite{sznitman}, and the proof for general $p \ge 1$ is a straightforward adaptation.

\begin{proposition} \label{pr:metawassersteincompact}
Let $(E,\rho)$ be a complete separable metric space. A subset $K$ of $\P^p(\P^p(E))$ is relatively compact if and only if $\{mP : P \in K\}$ is relatively compact in $\P^p(E)$ and
\begin{align}
\lim_{r \rightarrow \infty}\sup_{P \in K}\int_{\left\{\mu : \int_E\mu(dx)\rho^p(x,x_0) > r\right\}}P(d\mu)\int_E\mu(dx)\rho^p(x,x_0) = 0, \label{def:uimetawasserstein}
\end{align}
for some $x_0 \in E$.
\end{proposition}
\begin{proof}
Suppose first that $K$ is relatively compact. Note that 
\[
d^p_{E,p}(\mu,\delta_{x_0}) = \int_E\mu(dx)\rho^p(x,x_0),
\]
and thus the uniformly integrability \eqref{def:uimetawasserstein} holds by Proposition \ref{pr:wasserstein}(2). It is straightforward to show that $m : \P^p(\P^p(E)) \rightarrow \P^p(E)$ is continuous; indeed, suppose $P_n \rightarrow P$ in $\P^p(\P^p(E))$, and $\phi : E \rightarrow \R$ is continuous with $|\phi(x)| \le c(1+\rho^p(x,x_0))$ for some $c \ge 0$. Then
\[
\left|\int\phi\,d\mu\right| \le c\left(1+d^p_{E,p}(\mu,\delta_{x_0})\right),
\]
and thus Proposition \ref{pr:wasserstein}(3) implies
\[
\int\phi\,d[mP_n] = \int P_n(d\mu)\int\phi\,d\mu \rightarrow \int P(d\mu)\int\phi\,d\mu = \int\phi\,d[mP].
\]
Continuity of $m$ implies that $\{mP : P \in K\}$ is relatively compact.

Conversely, assume $\{mP : P \in K\}$ is relatively compact and \eqref{def:uimetawasserstein} holds. The uniform integrability assumption rewrites as
\[
\lim_{r \rightarrow \infty}\sup_{P \in K}\int_{\left\{\mu : d_{E,p}^p(\mu,\delta_{x_0}) \ge r\right\}}P(d\mu)d_{E,p}^p(\mu,\delta_{x_0}) = 0,
\]
so we need only to show that $K$ is tight, in light of Proposition \ref{pr:wasserstein}. Now suppose $P_n \in K$, and let $I_n := mP_n$. Define $\psi(x) := 1 + \rho^p(x,x_0)$. Relative compactness of $I_n$ in $\P^p(E)$ implies that
\begin{align*}
\lim_{r \rightarrow\infty}\sup_n\int_{\{\psi \ge r\}}\psi\,dI_n &= 0.
\end{align*}
Thus, for each $\epsilon > 0$ there exist $r(\epsilon) > 0$ and a compact set $K_\epsilon \subset E$ such that
\begin{align*}
\sup_nI_n(K_\epsilon^c) \le \epsilon/2, \quad \sup_n\int_{\{\psi > r(\epsilon)\}}\psi\,dI_n \le \epsilon/2.
\end{align*}
Now fix $\epsilon > 0$, and for each $k$ define
\begin{align*}
C_k = \left\{\mu \in \P^p(E) : \mu(K_{\epsilon 2^{-k}/k}^c) \le 1/k, \text{ and } \int_{\{\psi > r(\epsilon 2^{-k}/k)\}}\psi\,d\mu \le 1/k\right\}.
\end{align*}
Markov's inequality implies
\begin{align*}
P_n(C_k^c) &\le P_n\left\{\mu : \mu(K_{\epsilon 2^{-k}/k}^c) > 1/k\right\} + P_n\left\{\mu : \int_{\{\psi > r(\epsilon 2^{-k}/k)\}}\psi\,d\mu > 1/k\right\} \\
	&\le kI_n(K_{\epsilon 2^{-k}/k}^c) + k \int_{\{\psi > r(\epsilon 2^{-k}/k)\}}\psi\,dI_n \\
	&\le 2^{-k}\epsilon,
\end{align*}
and thus $P_n(\bigcup_{k\ge 1}C_k^c) \le \epsilon$. Since $1_{K_\eta^c}$ and $\psi1_{\psi > \eta}$ are lower semicontinuous on $E$ for each $\eta > 0$, it follows from Fatou's lemma that each $C_k$ is closed. Thus $\bigcap_{k\ge 1}C_k$ is compact, and $P_n$ are tight.
\end{proof}

\begin{corollary} \label{co:ptight}
Let $(E,\rho)$ be a complete separable metric space. Suppose $K \subset \P^p(\P^p(E))$ is such that $\{mP : P \in K\} \subset \P(E)$ is tight and
\[
\sup_{P \in K}\int mP(dx)\rho^{p'}(x,x_0) < \infty, \text{ for some } p' > p.
\]
Then $K$ is relatively compact.
\end{corollary}
\begin{proof}
The assumption along with Jensen's inequality imply
\[
\sup_{P \in K}\int P(d\mu)\left(\int\rho^{p}(x,x_0)\mu(dx)\right)^{p'/p} < \infty.
\]
and the uniform integrability condition \eqref{def:uimetawasserstein} of Proposition \ref{pr:metawassersteincompact}.
\end{proof}

Finally, we specialize the last result to a particular space of interest. As usual, assume $A$ is a closed subset of a Euclidean space, and define $(\V[A],d_{\V[A]})$ as in Section \ref{se:relaxed}. Endow $\Omega[A] = \V[A] \times \C^d$ with the metric formed by adding the metrics of $\V[A]$ and $\C^d$:
\[
((q,x),(q',x')) \mapsto \|x - x'\|_T + d_{\V[A]}(q,q').
\]

\begin{proposition} \label{pr:omegacompact}
Suppose $K \subset \P(\Omega[A])$ is such that $\{P \circ X^{-1} : P \in K\}$ is tight in $\P(\C^d)$ and 
\[
\sup_{P \in K}\E^P\left[\|X\|^{p'}_T + \int_0^T|\Lambda_t|^{p'}dt\right] < \infty.
\]
Then $K$ is relatively compact in $\P^p(\Omega[A])$.
\end{proposition}
\begin{proof}
It is immediate that $\{P \circ X^{-1} : P \in K\}$ is relatively compact in $\P^p(\C^d)$. For $P \in K$ define $m_0 P \in \V$ by
\[
m_0 P(C) := \int_{\Omega[A]} P(dq,dx)q(C).
\]
Since $A$ is a subset of a Euclidean space, Markov's inequality implies that $\{m_0 P : P \in K\} \subset \P(\V[A])$ is tight. Conclude from Corollary \ref{co:ptight} that $\{P \circ \Lambda^{-1} : P \in K\}$ is relatively compact in $\P^p(\V[A])$. The claim now follows from Lemma \ref{le:productrelcompactness} and Corollary \ref{co:ptight}.
\end{proof}

\begin{proposition} \label{pr:itocompact}
Let $d$ be a positive integer, and fix $c > 0$, $p' > p \ge 1 \vee p_\sigma$, $p_\sigma \in [0,2]$, and $\lambda \in \P^{p'}(\R^d)$. Let $\Q_c \subset \P(\Omega[A])$ be the set of laws $P \circ (\Lambda,X)^{-1}$ of $\Omega[A]$-valued random variables $(\Lambda,X)$ defined on some filtered probability space $(\Omega,\F_t,P)$ satisfying:
\begin{enumerate}
\item $dX_t = \int_Ab(t,X_t,a)\Lambda_t(da)dt + \int_A\sigma(t,X_t,a)N(da,dt)$.
\item $N = (N^1,\ldots,N^m)$ are orthogonal $\F_t$-martingale measures on $A \times [0,T]$ with intensity $\Lambda_t(da)dt$.
\item $\sigma : [0,T] \times \R^d \times A \rightarrow \R^{d \times d}$ and $b : [0,T] \times \R^d \times A \rightarrow \R^d$ are jointly measurable.
\item $X_0$ has law $\lambda$ and is $\F_0$-measurable.
\item For each $(t,x,a) \in [0,T] \times \R^d \times A$,
\begin{align*}
|b(t,x,a)| &\le c\left(1 + |x| + |a|\right), \\
|\sigma\sigma^\top(t,x,a)| &\le c\left(1 + |x|^{p_\sigma} + |a|^{p_\sigma}\right).
\end{align*}
\item Lastly,
\[
\E^P\left[|X_0|^{p'} + \int_0^Tdt|\Lambda_t|^{p'}\right] \le c.
\]
\end{enumerate}
(That is, we vary over $\sigma$, $b$, and the probability space of definition.) Then $\Q_c$ is relatively compact in $\P^p(\Omega[A])$.
\end{proposition}
\begin{proof}
For each $P \in \Q_c$ with corresponding probability space $(\Omega,\F_t,P)$ and coefficients $b$, $\sigma$, standard estimates as in Lemma \ref{le:stateestimate} yield
\begin{align*}
\E^P\|X\|^{p'}_T \le C\E^P\left[1 + |X_0|^{p'} + \int_0^Tdt|\Lambda_t|^{p'}\right].
\end{align*}
where $C > 0$ does not depend on $P$. Hence assumption (6) implies
\begin{align}
\sup_{P \in \Q_c}\E^P\|X\|^{p'}_T \le C(1 + c) < \infty. \label{pf:itocompact1}
\end{align}
The result will follow immediately from \eqref{pf:itocompact1} and Proposition \ref{pr:omegacompact} if we show that $\{P \circ X^{-1} : P \in \Q_c\} \subset \P(\C^d)$ is tight. To check this, we will verify Aldous' criterion for tightness \cite[Theorem 16.10]{billingsley-convergence},
\begin{align}
\lim_{\delta\downarrow 0}\sup_{P \in \Q_c}\sup_\tau \E^P|X_{(\tau + \delta) \wedge T} - X_\tau|^p = 0, \label{pf:itocompact2}
\end{align}
where the innermost supremum is over stopping times $\tau$ valued in $[0,T]$. The Burkholder-Davis-Gundy inequality implies that there exists $C' > 0$ such that, for each $P \in \Q_c$ and each $\tau$,
\begin{align*}
\E^P|X_{(\tau + \delta) \wedge T} - X_\tau|^p &\le C'\E^P\left[\left|\int_\tau^{(\tau + \delta) \wedge T}dt\int_A\Lambda_t(da)b(t,X_t,a)\right|^p\right] \\
	&\quad + C'\E^P\left[\left(\int_\tau^{(\tau + \delta) \wedge T}dt\int_A\Lambda_t(da)|\sigma\sigma^\top(t,X_t,a)|\right)^{p/2} \right] \\
	&\le C'\E^P\left[\left|\int_\tau^{(\tau + \delta) \wedge T}dt\int_A\Lambda_t(da)c(1 + \|X\|_T + |a|)\right|^p\right] \\
	&\quad + C'\E^P\left[\left(\int_\tau^{(\tau + \delta) \wedge T}dt\int_A\Lambda_t(da)c(1 + \|X\|^{p_\sigma}_T + |a|^{p_\sigma})\right)^{p/2} \right].
\end{align*}
Now note that if $1 \le p < 2$ and $x,y \ge 0$ then $(x+y)^{p/2}\le x^{p/2} + y^{p/2}$, and if $p \ge 2$ then $(x+y)^{p/2} \le 2^{p/2-1}(x^{p/2}+y^{p/2})$. In either case, we find another constant $C''$ such that, for each $P \in \Q_c$ and each $\tau$,
\begin{align}
\E^P|X_{(\tau + \delta) \wedge T} - X_\tau|^p  &\le C''\E^P\left[\left|\delta c(1+\|X\|_T)\right|^p + c^p\int_\tau^{(\tau + \delta) \wedge T}|\Lambda_t|^pdt\right] \nonumber \\
	&\quad + C''\E^P\left[\left|\delta c(1+\|X\|_T^{p_\sigma})\right|^{p/2} + \left|c\int_\tau^{(\tau + \delta) \wedge T}|\Lambda_t|^{p_\sigma}dt\right|^{p/2}\right] \label{pf:itocompact3}
\end{align}
The first term of each line poses no problems, in light of \eqref{pf:itocompact1}; that is, since $p_\sigma \le 2$ and $p < p'$,
\[
\lim_{\delta\downarrow 0}\sup_{P \in \Q_c}\sup_\tau \E^P\left[\left|\delta c(1+\|X\|_T)\right|^p + \left|\delta c(1+\|X\|_T^{p_\sigma})\right|^{p/2}\right] = 0.
\]
On the other hand, note that 
\[
\sup_{P \in \Q_c}\E^P\int_0^Tdt|\Lambda_t|^{p'} \le c < \infty,
\]
by assumption. It follows that for any $\gamma \in [0,p')$,
\[
\lim_{\delta \downarrow 0}\sup_{P \in \Q_c}\sup_\tau\E^P\int_\tau^{(\tau + \delta) \wedge T}|\Lambda_t|^\gamma dt = 0,
\]
and in particular this holds for $\gamma = p$. Hence, if $p \ge 2$ then Jensen's inequality along with $p_\sigma \le 2$ implies
\[
\lim_{\delta\downarrow 0}\sup_{P \in \Q_c}\sup_\tau\E^P\left|\int_\tau^{(\tau + \delta) \wedge T}|\Lambda_t|^{p_\sigma}dt\right|^{p/2} \le \lim_{\delta\downarrow 0}\sup_{P \in \Q_c}\sup_\tau\E^P\left|\int_\tau^{(\tau + \delta) \wedge T}|\Lambda_t|^pdt\right| = 0,
\]
On the other hand, if $p < 2$, then Jensen's inequality in the other direction implies
\[
\lim_{\delta\downarrow 0}\sup_{P \in \Q_c}\sup_\tau\E^P\left|\int_\tau^{(\tau + \delta) \wedge T}|\Lambda_t|^{p_\sigma}dt\right|^{p/2} \le \lim_{\delta\downarrow 0}\sup_{P \in \Q_c}\sup_\tau\left(\E^P\int_\tau^{(\tau + \delta) \wedge T}|\Lambda_t|^{p_\sigma}dt\right)^{p/2} = 0,
\]
since $p_\sigma \le p < p'$. Putting this together and returning to \eqref{pf:itocompact3} proves \eqref{pf:itocompact2}.
\end{proof}

\end{document}